\documentclass{amsart}

\usepackage{amsmath}
\usepackage{amsthm}
\usepackage{amsfonts}
\usepackage{dsfont}
\usepackage{color}
\usepackage{enumerate}
\usepackage{tikz-cd}
\usepackage[margin=42mm]{geometry}
\usepackage{etoolbox}
\makeatletter
\patchcmd{\@startsection}
{\@afterindenttrue}
{\@afterindentfalse}
{}{}
\makeatother
\newcommand{\reals}{\mathbb{R}}

\newcommand{\complex}{\mathbb{C}}

\newcommand{\integers}{\mathbb{Z}}

\newcommand{\bracketb}[1]{\Big[#1\Big]}

\newcommand{\angles}[1]{\left\langle #1 \right\rangle}

\newcommand{\paraa}[1]{\big(#1\big)}
\newcommand{\parab}[1]{\Big(#1\Big)}

\newcommand{\sgn}{\operatorname{sgn}}
\newcommand{\im}{\operatorname{im}}

\newcommand{\spacearound}[1]{\quad#1\quad}
\newcommand{\equivalent}{\spacearound{\Leftrightarrow}}
\renewcommand{\implies}{\spacearound{\Rightarrow}}

\newtheorem{theorem}{Theorem}[section]
\newtheorem{corollary}[theorem]{Corollary}
\newtheorem{lemma}[theorem]{Lemma}
\newtheorem{proposition}[theorem]{Proposition}

\theoremstyle{definition}
\newtheorem{definition}[theorem]{Definition}
\theoremstyle{remark}

\numberwithin{equation}{section}

\newcommand{\A}{\mathcal{A}}
\newcommand{\K}{\mathcal{K}}

\renewcommand{\mid}{\mathds{1}}

\newcommand{\KN}{\K_N}
\newcommand{\KaN}{\K^\alpha_N}
\renewcommand{\d}{\partial}
\newcommand{\dt}{\tilde{\d}}
\newcommand{\Der}{\operatorname{Der}}
\newcommand{\qand}{\quad\text{and}\quad}
\newcommand{\qqand}{\qquad\text{and}\qquad}
\newcommand{\lb}{\bar{\lambda}}
\newcommand{\mub}{\bar{\mu}}
\newcommand{\g}{\mathfrak{g}}
\newcommand{\ginn}{\g_{\textrm{inn}}}

\renewcommand{\dh}{\hat{\partial}}
\newcommand{\Omegag}[1]{\Omega^{#1}_{\g}}
\newcommand{\Omegabg}[1]{\bar{\Omega}^{#1}_{\g}}
\newcommand{\Omegaoneg}{\Omega^1_{\g}}
\renewcommand{\emph}[1]{\textit{#1}}
\newcommand{\otimesKN}{\otimes_{\KN}}
\newcommand{\otimesC}{\otimes_{\complex}}
\newcommand{\hh}{\hat{h}}
\newcommand{\hhi}{\hh^{-1}}
\newcommand{\Ttwotheta}{T^2_{\theta}}
\renewcommand{\Re}{\operatorname{Re}}
\renewcommand{\Im}{\operatorname{Im}}

\title[Noncommutative Riemannian geometry of Kronecker algebras]{Noncommutative Riemannian geometry\\ of Kronecker algebras}
\author{Joakim Arnlind}
\address[Joakim Arnlind]{Dept. of Math.\\
Link\"oping University\\
581 83 Link\"oping\\
Sweden}
\email{joakim.arnlind@liu.se}

\thanks{}

\subjclass[2000]{}

\begin{document}

\begin{abstract}
  We study aspects of noncommutative Riemannian geometry of the path
  algebra arising from the Kronecker quiver with $N$ arrows. To start
  with, the framework of derivation based differential calculi is
  recalled together with a discussion on metrics and bimodule
  connections compatible with the $\ast$-structure of the algebra. As
  an illustration, these concepts are applied to the noncommutative
  torus where examples of torsion free and metric (Levi-Civita)
  connections are given.

  In the main part of the paper, noncommutative geometric aspects of (generalized)
  Kronecker algebras are considered. The structure of derivations and
  differential calculi is explored, and torsion free bimodule
  connections are studied together with their compatibility with
  hermitian forms, playing the role of metrics on the module of
  differential forms. Moreover, for several different choices of Lie algebras of
  derivations, non-trivial Levi-Civita connections are constructed.
\end{abstract}

\maketitle

\tableofcontents

%\newpage

\section{Introduction}

\noindent
Over the last 15 years, at lot of progress has been made in
understanding the Riemannian aspects of noncommutative geometry from
several different perspectives. For instance, the curvature of the
metric, identified from the heat kernel expansion of the Dirac
operator, has been computed
\cite{fk:scalarCurvature,cm:modularCurvature} and related to the
Gauss-Bonnet theorem \cite{ct:gaussBonnet,fk:gaussBonnet}. From a more
(Hopf) algebraic point of view, metric and connections on bimodules
(see e.g. \cite{ac:ncgravitysolutions}
\cite{as:noncommutative.connections.twists}
\cite{bm:Quantum.Riemannian.geometry} \cite{ail:lc.quantum.spheres})
have been introduced and studied from different points of view.

A fundamental aspect of classical Riemannian geometry is the existence
of a unique metric and torsion free connection -- the Levi-Civita
connection. In noncommutative geometry, the definitions of metric
compatibility and torsion free depend on the context, and the existence
and uniqueness of the Levi-Civita connection has been studied from
many different perspectives (see
e.g. \cite{bm:starCompatibleConnections,r:leviCivita,ps:on.nc.lc.connections,al:projections.nc.cylinder,w:braided.cartan.calculi,a:levi-civita.class.nms}).
For instance, starting from a derivation based approach to
noncommutative differential geometry (see
e.g. \cite{dv:calculDifferentiel,dvm:connections.central.bimodules})
the framework of real metric calculi was developed in
\cite{aw:curvature.three.sphere} (and further developed in
\cite{atn:minimal.embeddings.morphisms}) and a Gauss-Bonnet theorem
was proven for the noncommutative 4-sphere \cite{aw:cgb.sphere}. In
this context, one can prove a general uniqueness theorem for torsion
free and metric connections. One the other hand, if further conditions
are abandoned, like the reality conditions of a real metric calculus,
one expects there to be many torsion free and metric connections (see
e.g. \cite{a:levi-civita.class.nms}).

Concrete and illustrative examples have traditionally been of great use in
noncommutative geometry; both as a mean to understand abstract
concepts, but also as an inspiration for new results. In the present
paper, we study differential calculi and Levi-Civita connections of
the path algebra of generalized Kronecker quivers for different
choices of Lie algebras of derivations. These algebras turn out to be
one dimensional, in the sense that there are no differential forms of
order greater than one; a result which is independent on the choice of
Lie algebra of derivations. Despite the low-dimensionality of the
calculus, the structure of the module of 1-forms is not trivial and we
show that one may construct many interesting examples of Levi-Civita
connections.

Section~\ref{sec:der.calculus} recalls the basic concepts of
derivation based differential calculus in order to fix the notation
and terminology for our purposes. As an illustration, these concepts
are applied to the noncommutative torus in Section~\ref{sec:nc.torus}
where left module connections, as well as bimodule connections, are
constructed on the module of differential forms.

In Section~\ref{sec:kronecker.alg} we study the path algebra of the
generalized Kronecker quiver, and derive properties of the algebra and
its differential calculi. Section~\ref{sec:LC.Omegaoneg} is devoted to
the construction of Levi-Civita connections on the module of 1-forms
for different choices of Lie algebras of derivations.

\section{Derivation based differential calculus}\label{sec:der.calculus}

Let us recall the construction of noncommutative differential forms in
a derivation based calculus. Let $\A$ be a unital associative $\ast$-algebra
over $\complex$, and let $\Der(\A)$ denote the set of derivations of
$\A$. We shall consider $\Der(\A)$ to be both a left and right $Z(\A)$-module, where
$Z(\A)$ denotes the center of $\A$, in the standard way;
i.e.
\begin{align*}
  &(z\cdot \d)(a)=z\d(a)\\
  &(\d\cdot z)(a)=\d(a)z=z\d(a)=(z\cdot \d)(a)
\end{align*}
for $z\in Z(\A)$, $a\in\A$ and $\d\in\Der(\A)$. For any Lie algebra
$\g\subseteq\Der(\A)$ we shall in the following assume that $\g$ is
also a $Z(\A)$-submodule of $\Der(\A)$.  Given $\g\subseteq\Der(\A)$,
one defines $\Omegabg{k}$ to be the set of $Z(\A)$-multilinear
alternating maps
\begin{align*}
  \omega:\underbrace{\g\times\cdots\times\g}_k\to\A,
\end{align*}
and sets
\begin{align*}
  \Omegabg{} = \bigoplus_{k\geq 0}\Omegabg{k}
\end{align*}
with $\Omegabg{0}=\A$. Furthermore, for $\omega\in\Omegabg{k}$ and
$\tau\in\Omegabg{l}$ one defines $\omega\tau\in\Omegabg{k+l}$ as
\begin{align*}
  (\omega\tau)(\d_1,\ldots,\d_{k+l}) =
  \frac{1}{k!l!}\sum_{\sigma\in S_{k+l}}\sgn(\sigma)\omega(\d_{\sigma(1)},\ldots,\d_{\sigma(k)})
  \tau(\d_{\sigma(k+1)},\ldots,\d_{\sigma(k+l)}),
\end{align*}
where $S_N$ denotes the symmetric group on $N$ letters, and introduce
$d_k:\Omegabg{k}\to\Omegabg{k+1}$ as $d_0a(\d_0)=\d_0(a)$ for
$a\in\Omegabg{0}=\A$ and for $\omega\in\Omegabg{k}$ with $k\geq 1$
\begin{align*}
  d_k\omega(\d_0,\ldots,\d_{k}) =
  &\sum_{i=0}^k(-1)^i\d_i\paraa{\omega(\d_0,\ldots,\hat{\d}_{i},\ldots,\d_k)}\\
  &\qquad+\sum_{0\leq i<j\leq k}(-1)^{i+j}\omega\paraa{[\d_i,\d_j],\d_0,\ldots,\hat{\d}_i,\ldots,\hat{\d}_j,\ldots,\d_k},
\end{align*}
satisfying $d_{k+1}d_{k}=0$, where $\hat{\d}_i$ denotes the omission
of $\d_i$ in the argument. When there is no risk for confusion, we
shall omit the index $k$ and simply write
$d:\Omegabg{k}\to\Omegabg{k+1}$. For instance, if
$\omega\in\Omegabg{1}$ then
\begin{align*}
  d\omega(\d_0,\d_1) = \d_0\omega(\d_1)-\d_1\omega(\d_0)-\omega([\d_0,\d_1]). 
\end{align*}
Moreover, we note that the graded product rule is satisfied
\begin{align*}
  d(\omega\eta) = (d\omega)\eta + (-1)^k\omega d\eta 
\end{align*}
for $\omega\in\Omegabg{k}$ and $\eta\in\Omegabg{l}$. 

One can endow $\Omegabg{}$ with the structure of an $\A$-bimodule by
setting
\begin{align*}
  &(a\omega)(\d_1,\ldots,\d_k) = a\omega(\d_1,\ldots,\d_k)\\
  &(\omega a)(\d_1,\ldots,\d_k) = \omega(\d_1,\ldots,\d_k)a
\end{align*}
for $a\in\A$, and a $\ast$-structure may be introduced as
\begin{align*}
  \omega^\ast(\d_1,\ldots,\d_k) = \omega(\d_1^\ast,\ldots,\d_k^\ast)^\ast
\end{align*}
satisfying $(a\omega b)^\ast = b^\ast\omega^\ast a^\ast$, making
$\Omegabg{}$ into a $\ast$-bimodule with
\begin{align*}
  (\omega\eta)^\ast = (-1)^{kl}\eta^\ast\omega^\ast
\end{align*}
for $\omega\in\Omegabg{k}$ and $\eta\in\Omegabg{l}$. We say that
$\Omegabg{}$ is a differential calculus over $\A$. The differential
calculus is called connected if $da=0$ implies that
$a=\lambda\mid$ for some $\lambda\in\complex$.  The cohomology of a
differential calculus is introduced in the standard way
\begin{align*}
  H^n(\Omegabg{}) = \frac{\ker d_n}{\im d_{n-1}}.
\end{align*}
Let us now recall the concept of a connection on a left or right $\A$-module.
\begin{definition}
  Let $M$ be a left $\A$-module. A \emph{left connection on $M$} is a map
  $\nabla:\g\times M\to M$ such that
  \begin{align*}
    &\nabla_{\d}\paraa{m+m'} = \nabla_{\d}m + \nabla_{\d}m'\\
    &\nabla_{\d+\d'}m = \nabla_{\d}m + \nabla_{\d'}m\\
    &\nabla_{z\cdot \d}m = z\nabla_{\d}m\\
    &\nabla_{\d}(am) = a\nabla_{\d}m + (\d a)m
  \end{align*}
  for $m,m'\in M$, $\d,\d'\in\g$, $a\in\A$ and
  $z\in Z(\A)$.
\end{definition}

\begin{definition}
  Let $M$ be a right $\A$-module. A \emph{right connection on $M$} is a map
  $\nabla:\g\times M\to M$ such that
  \begin{align*}
    &\nabla_{\d}\paraa{m+m'} = \nabla_{\d}m + \nabla_{\d}m'\\
    &\nabla_{\d+\d'}m = \nabla_{\d}m + \nabla_{\d'}m\\
    &\nabla_{\d\cdot z}m = (\nabla_{\d}m)z\\
    &\nabla_{\d}(ma) = (\nabla_{\d}m)a + m(\d a)
  \end{align*}
  for $m,m'\in M$, $\d,\d'\in\g$, $a\in\A$ and
  $z\in Z(\A)$.
\end{definition}

\noindent
As in differential geometry, the curvature $R$ of a
connection $\nabla$ is defined as
\begin{align*}
  R(\d_1,\d_2)m = \nabla_{\d_1}\nabla_{\d_2}m-\nabla_{\d_2}\nabla_{\d_1}m
  -\nabla_{[\d_1,\d_2]}m.
\end{align*}

\noindent
Now, assume that $\nabla$ is both a left and right connection on the
$\A$-bimodule $M$. In particular, one has
\begin{align*}
  z\nabla_{\d}m = \nabla_{z\cdot \d}m = \nabla_{\d\cdot z}m = (\nabla_{\d}m)z
\end{align*}
for all $m\in M$, $\d\in\g$ and $z\in Z(\A)$. Hence, a natural setting
for bimodule connections is in the context of central bimodules.

\begin{definition}
  A \emph{central $\A$-bimodule} is an $\A$-bimodule $M$ such that
  \begin{align*}
    zm = mz
  \end{align*}
  for all $m\in M$ and $z\in Z(\A)$.
\end{definition}

\noindent
Note that it follows that $\Omegabg{}$ is a central bimodule since
\begin{align*}
  (z\omega)(\d_1,\ldots,\d_k)
  = z\omega(\d_1,\ldots,\d_k)
  = \omega(\d_1,\ldots,\d_k)z
  = (\omega z)(\d_1,\ldots,\d_k)
\end{align*}
for $z\in Z(\A)$.

\begin{definition}
  Let $M$ be a central $\A$-bimodule. If $\nabla$ is a both a left and
  a right connection on $M$ then $\nabla$ is called a \emph{bimodule
    connection on $M$}.
\end{definition}

\noindent
In case the module $M$ is a $\ast$-bimodule, one can ask that the
connection is compatible with the involution in the following sense.

\begin{definition}
  Let $\nabla:\g\times M\to M$ be a left connection on the
  $\ast$-bimodule $M$. The connection is called a left
  $\ast$-connection if
  \begin{equation*}
    \paraa{\nabla_\d m}^\ast = \nabla_{\d^\ast}m^\ast
  \end{equation*}
  for all $\d\in\g$ and $m\in M$.
\end{definition}

\begin{proposition}\label{prop:l.st.conn.bimodule.conn}
  If $\nabla$ is a left $\ast$-connection on a central $\ast$-bimodule
  $M$ then $\nabla$ is a bimodule connection on $M$.
\end{proposition}

\begin{proof}
  Since $\nabla$ is a left $\ast$-connection and $M$ is a $\ast$-bimodule one has
  \begin{align*}
    \nabla_\d (ma) &= \paraa{\nabla_{\d^\ast}a^\ast m^\ast}^\ast
    =\paraa{a^\ast\nabla_{\d^\ast}m^\ast+(\d^\ast a^\ast)m^\ast}^\ast\\
                   &=\paraa{\nabla_{\d^\ast}m^\ast}^\ast a + m(\d^\ast a^\ast)^\ast
                     = (\nabla_{\d}m)a + m(\d a)
  \end{align*}
  for $\d\in\g$, $m\in M$ and $a\in\A$, showing that $\nabla$ is also a right connection on $M$.
\end{proof}

\noindent
A left $\ast$-connection on a central $\ast$-bimodule will be called a
$\ast$-bimodule connection. The above lemma shows that bimodule
connections arise naturally on $\ast$-bimodules when requiring the
connection to be compatible with the $\ast$-structure.

In noncommutative geometry, a finitely generated projective module $M$
corresponds to a vector bundle and metrics on vector bundles can be
introduced as hermitian forms on $M$. Let us recall the definition.

\begin{definition}
  Let $M$ be a left $\A$-module. A \emph{left hermitian form on $M$}
  is a map $h:M\times M\to\A$ such that
  \begin{align*}
    &h(m_1+m_2,m_3) = h(m_1+m_2,m_3)\\
    &h(am_1,m_2) = ah(m_1,m_2)\\
    &h(m_1,m_2)^\ast = h(m_2,m_1)
  \end{align*}
  for all $m_1,m_2,m_3\in M$ and $a\in\A$.
\end{definition}

\noindent
Similarly, a right hermitian form on a right $\A$-module $M$ is a bilinear map satisfying
\begin{align*}
  h(m_1,m_2a) = h(m_1,m_2)a
\end{align*}
together with $h(m_1,m_2)^\ast=h(m_2,m_1)$ for $m_1,m_2\in M$ and
$a\in M$.  In the case of $\ast$-bimodules, one can combine the
concepts of left and right hermitian forms into $\ast$-bimodule forms.

\begin{definition}
  Let $M$ be a $\ast$-bimodule. A $\ast$-bimodule form on $M$ is
  a $\complex$-bilinear map $g:M\times M\to\A$ such that
  \begin{align}
    &g(m_1,m_2)^\ast = g(m_2^\ast,m_1^\ast)\label{eq:star.bimodule.form.starprop}\\
    &g(am_1,m_2) = ag(m_1,m_2)\label{eq:star.bimodule.form.leftlinear}
  \end{align}
  for all $m_1,m_2\in M$ and $a\in\A$.
\end{definition}

\noindent
Note that it follows from \eqref{eq:star.bimodule.form.starprop} and
\eqref{eq:star.bimodule.form.leftlinear} that
\begin{align}
  g(m_1,m_2 a) = g(m_1,m_2)a
\end{align}
for all $m_1,m_2\in\A$ and $a\in\A$.

Given a $\ast$-bimodule form $g$ on $M$ we note that
$h_L(m_1,m_2)=g(m_1,m_2^\ast)$ is a left hermitian form on $M$ and
$h_R(m_1,m_2)=g(m_1^\ast,m_2)$ is a right hermitian form on
$M$. Conversely, given a left hermitian form $h$ on $M$, one obtains a
$\ast$-bimodule form $g(m_1,m_2)=h(m_1,m_2^\ast)$ as well as a right
hermitian form $h_R(m_1,m_2)=h(m_1^\ast,m_2^\ast)$.

As in Riemannian geometry, one is interested in connections that
preserve the metric. Hence, one makes the following definitions.

\begin{definition}
  Let $M$ be a left $\A$-module. A left connection $\nabla$ on $M$ is
  compatible with a left hermitian form $h$ if
  \begin{align*}
    &\d h(m_1,m_2) = h(\nabla_\d m_1,m_2) + h(m_1,\nabla_{\d^\ast}m_2)
  \end{align*}
  for all $m_1,m_2\in M$ and $\d\in\g$.
\end{definition}

\begin{definition}
  Let $M$ be a right $\A$-module. A right connection $\nabla$ on $M$ is
  compatible with a right hermitian form $h$ if
  \begin{align*}
    &\d h(m_1,m_2) = h(\nabla_{\d^\ast} m_1,m_2) + h(m_1,\nabla_{\d}m_2)
  \end{align*}
  for all $m_1,m_2\in M$ and $\d\in\g$.
\end{definition}

\begin{definition}
  Let $M$ be a $\ast$-bimodule and let $g$ be a $\ast$-bimodule
  form on $M$. A left (right) connection $\nabla:\g\times M\to M$ is
  compatible with $g$ if
  \begin{align*}
    \d g(m_1,m_2) = g(\nabla_\d m_1,m_2) + g(m_1,\nabla_{\d}m_2)
  \end{align*}
  for all $\d\in\g$ and $m_1,m_2\in M$.
\end{definition}

\noindent
For $\ast$-bimodule connections, the above concepts of compatibility
are equivalent in the following sense.

\begin{proposition}\label{prop:bimodule.comp.gives.lr.comp}
  Let $M$ be a $\ast$-bimodule and let $g$ be a $\ast$-bimodule
  form on $M$ and set $h_L(m_1,m_2) = g(m_1,m_2^\ast)$ and
  $h_R(m_1,m_2) = g(m_1^\ast,m_2)$.  If $\nabla:\g\times M\to M$ is a
  $\ast$-bimodule connection compatible with $g$ then
  \begin{align*}
    &\d h_L(m_1,m_2) = h_L(\nabla_\d m_1,m_2) + h_L(m_1,\nabla_{\d^\ast}m_2)\\
    &\d h_R(m_1,m_2) = h_R(\nabla_{\d^\ast} m_1,m_2) + h_R(m_1,\nabla_{\d}m_2)
  \end{align*}
  for all $\d\in\g$ and $m_1,m_2\in M$. Conversely, if $\nabla$ is a
  $\ast$-bimodule connection compatible with a left hermitian form
  $h_L$ then $\nabla$ is compatible with the right hermitian form
  $h_R(m_1,m_2)=h_L(m_1^\ast,m_2^\ast)$ as well as the $\ast$-bimodule
  form $g(m_1,m_2)=h_L(m_1,m_2^\ast)$.
\end{proposition}

\begin{proof}
  Assume that $\nabla$ is a $\ast$-bimodule connection compatible with the $\ast$-bimodule form $g$.
  One easily checks that
  \begin{align*}
    h_L(\nabla_\d m_1,m_2)
    &+ h_L(m_1,\nabla_{\d^\ast}m_2)
      = g(\nabla_{\d}m_1,m_2^\ast)+g\paraa{m_1,(\nabla_{\d^\ast}m_2)^\ast}\\
    &= g(\nabla_{\d}m_1,m_2^\ast) + g(m_1,\nabla_{\d}m_2^\ast)
      = \d g(m_1,m_2^\ast) = \d h_L(m_1,m_2),
  \end{align*}
  using that $\nabla$ is a $\ast$-connection compatible with $g$. The
  proof that $\nabla$ is compatible with $h_R$ is analogous. Now,
  assume that $\nabla$ is a $\ast$-bimodule connection compatible with
  the left hermitian form $h_L$. Then
  \begin{align*}
    h_R(\nabla_{\d^\ast}m_1,m_2)
    &+ h_R(m_1,\nabla_{\d}m_2)
      = h_L\paraa{(\nabla_{\d^\ast}m_1)^\ast,m_2^\ast}
      +h_L\paraa{m_1^\ast,(\nabla_{\d}m_2)^\ast}\\
    &= h_L(\nabla_{\d}m_1^\ast,m_2^\ast) + h_L(m_1^\ast,\nabla_{\d^\ast}m_2^\ast)
      = \dh_L(m_1^\ast,m_2^\ast) = \d h_R(m_1,m_2)
  \end{align*}
  and
  \begin{align*}
    g(\nabla_{\d}m_1,m_2)
    &+g(m_1,\nabla_{\d}m_2)
      = h_L(\nabla_{\d}m_1,m_2^\ast) + h_L\paraa{m_1,(\nabla_{\d}m_2)^\ast}\\
    &=h_L(\nabla_{\d}m_1,m_2^\ast) + h_L(m_1,\nabla_{\d^\ast}m_2^\ast)
      =\d h_L(m_1,m_2^\ast) = \d g(m_1,m_2),
  \end{align*}
  showing that $\nabla$ is indeed compatible with $h_R$ and $g$.
\end{proof}

\noindent 
Proposition~\ref{prop:bimodule.comp.gives.lr.comp} together with
Proposition~\ref{prop:l.st.conn.bimodule.conn} show that for
$\ast$-connections on a $\ast$-bimodule, the concepts of
left-/right-/bimodule connections coincide together with their
corresponding concepts of compatibility with hermitian forms.

Although the above definitions give a natural concept of metric
compatibility for a connection, there is in general
not a unique concept of torsion for arbitrary modules (although one
can define torsion via an \emph{anchor map}, see
\cite{aw:curvature.three.sphere}). However, for $\Omegabg{1}$ one
introduces torsion in analogy with differential geometry.

\begin{definition}
  The \emph{torsion} of a left (right) connection
  $\nabla$ on $\Omegabg{1}$ is given by the map
  $T:\Omegabg{1}\times\g\times\g\to\A$, defined by
  \begin{align}\label{eq:def.torsion}
    T_\omega(\d,\d') = (\nabla_{\d}\omega)(\d')-(\nabla_{\d'}\omega)(\d)
    -d\omega(\d,\d').
  \end{align}
  The connection is called \emph{torsion free} if $T_\omega(\d,\d')=0$ for all
  $\d,\d'\in\g$ and $\omega\in\Omegabg{1}$.
\end{definition}

\noindent
It is easy to check that $T_{a\omega}(\d,\d')=aT_\omega(\d,\d')$ for a
left connection (and $T_{\omega a}(\d,\d')=T_{\omega}(\d,\d')a$ for a
right connection), implying that the induced map
$T(\d,\d'):\Omegabg{1}\to\A$ is a left (right) module
homomorphism.

Now, we are ready to discuss torsion free connections
compatible with a hermitian form on $\Omegabg{1}$, so called
\emph{Levi-Civita connections}.

\begin{definition}
  Let $h$ be a left (right) hermitian form on $\Omegabg{1}$. A \emph{left
     (right) Levi-Civita connection $\nabla$ on $\Omegabg{1}$ with respect to
    $h$} is a torsion free left (right) connection on $\Omegabg{1}$ compatible
  with $h$.
\end{definition}

\begin{definition}
  Let $g$ be a $\ast$-bimodule form on $\Omegabg{1}$. A
  \emph{$\ast$-bimodule Levi-Civita connection on $\Omegabg{1}$ with
    respect to $g$} is a torsion free $\ast$-bimodule connection on
  $\Omegabg{1}$ compatible with $g$.
\end{definition}

\noindent
It follows from Proposition~\ref{prop:bimodule.comp.gives.lr.comp}
that a $\ast$-bimodule Levi-Civita connection with respect to $g$ is also a left and right
Levi-Civita connection with respect to the associated left and right hermitian forms.

\subsection{Restricted calculi}

In the following, we shall mostly be interested in so called restricted calculi
$\Omegag{}\subseteq\Omegabg{}$, which is generated (as a left module) in degree one by
$da$ for $a\in\A$. That is, $\Omegag{k}$ is generated by elements of the form
\begin{align*}
  a_0da_1da_2\cdots da_k
\end{align*}
for $a_0,a_1,\ldots,a_k\in\A$, acting as
\begin{align*}
  a_0da_1\cdots da_k(\d_1,\ldots,\d_k) =
  \sum_{\sigma\in S_k}\sgn(\sigma)a_0da_1(\d_{\sigma(1)})\cdots da_k(\d_{\sigma(k)})
\end{align*}
for $\d_1,\ldots,\d_k\in\g$.  One readily checks that $\Omegag{k}$ is
closed under the right action of $\A$ by repeatedly using
$(da)b=d(ab)-adb$ for $a,b\in\A$, giving
\begin{align*}
  da_1da_2\cdots da_k\cdot a_{k+1}
  &= \sum_{l=1}^k(-1)^{k-l}da_1\cdots d(a_la_{l+1})\cdots da_k da_{k+1}\\
      &\qquad+(-1)^ka_1da_2\cdots da_k da_{k+1}
\end{align*}
which is clearly an element of $\Omegag{k}$. Furthermore, one can show that
\begin{align*}
  d\paraa{a_0da_1da_2\cdots da_k} = da_0da_1da_2\cdots da_k
\end{align*}
which is an element of $\Omegag{k+1}$. Thus, $\Omegag{}$ is a
differential subalgebra of $\Omegabg{}$, and the introduced concepts
of hermitian forms, connections and torsion apply equally well to this
subalgebra.

\section{The noncommutative torus}\label{sec:nc.torus}

\noindent
Let us illustrate the above concepts by considering the noncommutative
torus. The noncommutative torus $\Ttwotheta$ is a $\ast$-algebra
generated by unitary $U,V$ satisfying $VU=qUV$ with
$q=e^{2\pi i\theta}$. We shall assume that $\theta$ is irrational,
implying that the center of $\Ttwotheta$ is trivial,
i.e. $Z(\Ttwotheta)=\{\lambda\mid:\lambda\in\complex\}$. Any element
$a\in\Ttwotheta$ can be written as
\begin{align*}
  a = \sum_{k,l\in\integers}a_{kl}U^kV^l
\end{align*}
where $a_{kl}\in\complex$, and one introduces the standard derivations
$\d_1,\d_2$, defined by
\begin{align*}
  &\d_1U = iU\qquad \d_1V = 0\\
  &\d_2U = 0 \qquad \d_2V = iV
\end{align*}
and it follows that $[\d_1,\d_2]=0$ as well as $\d_a^\ast=\d_a$ for
$a=1,2$. Since $Z(\A)$ is trivial, the complex (abelian) Lie algebra $\g$
generated by $\d_1,\d_2$ is a $Z(\A)$-module.

The restricted differential calculus $\Omegag{}$ is generated in first
degree by $dU$ and $dV$, and we note that $\{dU,dV\}$ is a basis for
$\Omegaoneg$:
\begin{align*}
  &adU+bdV = 0\implies
  \begin{cases}
    adU(\d_1)+bdV(\d_1) = 0\\
    adU(\d_2)+bdV(\d_2) = 0
  \end{cases}\implies\\
  &
    \begin{cases}
      iaU = 0\\
      ibV = 0
    \end{cases}\implies
        a=b=0.
\end{align*}
Thus, $\Omegaoneg$ is a free module, and the bimodule structure can be summarized as follows:
\begin{equation}
  \begin{split}
    (dU)U &= UdU \qquad (dU)V = q^{-1}V(dU)\\
    (dV)V &= VdV \qquad (dV)U = qUdV.            
  \end{split}
\end{equation}
For instance, checking that $(dU)V=q^{-1}VdU$ amounts to showing that
\begin{align*}
  &\paraa{(dU)V}(\d_1) = (\d_1U)V = iUV = iq^{-1}VU=q^{-1}(VdU)(\d_1)\\
  &\paraa{(dU)V}(\d_2) = (\d_2U)V = 0 = q^{-1}(VdU)(\d_2).
\end{align*}
Moreover, one can easily check that $\Omegag{}$ is a connected
calculus, since
\begin{align*}
  &\d_1\paraa{a_{kl}U^kV^l} = 0\implies ika_{kl}U^kV^l = 0\implies
  a_{kl} = 0\text{ for }k\neq 0.\\
  &\d_2\paraa{a_{kl}U^kV^l} = 0\implies ila_{kl}U^kV^l = 0\implies
  a_{kl} = 0\text{ for }l\neq 0,
\end{align*}
and the only possible nonzero coefficient is given by $a_{00}$, and
one concludes that
\begin{align*}
  da = 0\equivalent
  \begin{cases}
    \d_1a = 0\\
    \d_2a = 0
  \end{cases}\equivalent
  a = \lambda\mid
\end{align*}
for $\lambda\in\complex$.  A slightly more convenient basis of
$\Omegaoneg$ is given by
\begin{align*}
  \omega^1 = \omega = -iU^{-1}dU\qand
  \omega^2 = \eta = -iV^{-1}dV
\end{align*}
satisfying
\begin{align*}
  &\omega^a(\d_b) = \delta^a_b\mid\qquad
    (\omega^a)^\ast = \omega^a\qquad [\omega^a,f] = 0\\
  &\omega\eta = -\eta\omega\qquad d\omega=d\eta = 0
\end{align*}
for all $f\in\Ttwotheta$ and $a,b=1,2$. We note that $\omega\eta$ is a basis of $\Omegag{2}$ since
\begin{align*}
  &a\omega\eta = 0\implies
  a\omega\eta(\d_1,\d_2) = 0\implies\\
  &a\paraa{\omega(\d_1)\eta(\d_2)-\omega(\d_2)\eta(\d_1)} = 0
  \implies a=0.
\end{align*}
Let us now compute the cohomology of the differential calculus $\Omegag{}$.

\begin{proposition}
  The cohomology of $\Omegag{}$ is given by
  \begin{align*}
    H^0(\Omegag{}) = \complex\qquad
    H^1(\Omegag{}) = \complex^2\qquad
    H^2(\Omegag{}) = \complex.
  \end{align*}
\end{proposition}

\begin{proof}
  Since $\Omegag{}$ is connected one finds that
  \begin{align*}
    H^0(\Omegag{}) = \{a\in\Ttwotheta: da = 0\} = \{\lambda\mid:\lambda\in\complex\}=\complex.
  \end{align*}
  Let us now consider $H^2(\Omegag{})$. First, one notes that
  \begin{align*}
    d(U^k) = ikU^k\omega\qand
    d(V^k) = ikV^k\eta
  \end{align*}
  giving $d(U^kV^l)=iU^kV^l(k\omega + l\eta)$.  Since $\dim(\g)=2$
  any element of $\Omegag{2}$ is closed, so to compute
  $H^2(\Omegag{})$ one has to find out which 2-forms that are not
  exact. An exact 2-form can be written as $d\rho$ for
  $\rho = a\omega + b\eta\in\Omegag{1}$; writing $a=a_{kl}U^kV^l$ and
  $b=b_{kl}U^kV^l$, one obtains 
  \begin{equation}\label{eq:drho}
    \begin{split}      
    d\rho &= (da)\omega + (db)\eta =
    \sum_{k,l\in\integers}ia_{kl}U^kV^l\paraa{k\omega + l\eta}\omega
    +\sum_{k,l\in\integers}ib_{kl}U^kV^l\paraa{k\omega + l\eta}\eta\\
    &=
    \sum_{k,l\in\integers}i\paraa{kb_{kl}-la_{kl}}U^kV^l\omega\eta,
    \end{split}
  \end{equation}
  implying that $(c_{kl}U^kV^l)\omega\eta$ is exact if and only if
  $c_{00}=0$. Hence,
  \begin{align*}
    H^2(\Omegag{}) = \{\lambda\omega\eta : \lambda\in\complex\} = \complex.
  \end{align*}
  Finally, let us consider $H^1(\Omegag{})$. Let
  $\rho=a\omega + b\eta$ be a closed $1$-form, giving (using \eqref{eq:drho})
  \begin{align*}
    \sum_{k,l\in\integers}i\paraa{kb_{kl}-la_{kl}}U^kV^l\omega\eta = 0\implies
    \sum_{k,l\in\integers}i\paraa{kb_{kl}-la_{kl}}U^kV^l = 0
  \end{align*}
  since $\omega\eta$ is a basis of $H^2(\Omegag{})$. Thus, 
  \begin{equation}\label{eq:drhoab}
    d\rho = 0\equivalent kb_{kl} = la_{kl}
  \end{equation}
  for all $k,l\in\integers$, giving $a_{0k}=b_{k0}=0$ for $k\neq
  0$. Such coefficients can be parametrized by $r_{kl}\in\complex$ as
  \begin{equation}\label{aklbklr}
    b_{kl} = lr_{kl}\qand
    a_{kl} = kr_{kl}
  \end{equation}
  for $(k,l)\neq(0,0)$, leaving $a_{00}$ and $b_{00}$ arbitrary.  On
  the other hand, exact 1-forms are given by
  \begin{align*}
    \rho = dc = d\sum_{k,l\in\integers}c_{kl}U^kV^l
    = \sum_{k,l\in\integers}ikc_{kl}U^kV^l\omega+\sum_{k,l\in\integers}ilc_{kl}U^kV^l\eta,
  \end{align*}
  and comparing with \eqref{aklbklr} one concludes that the
  closed 1-forms that are not exact can be represented by
  $\rho = a_{00}\omega + b_{00}\eta$ for $a_{00},b_{00}\in\complex$;
  hence
  \begin{equation}
    H^1(\Omegag{}) = \{\lambda\omega + \mu\eta:\lambda,\mu\in\complex\} \simeq \complex^2.\qedhere
  \end{equation}
\end{proof}

\subsection{Levi-Civita connections}

Let us now consider connections on $\Omegaoneg$. Since
$\{\omega^1,\omega^2\}$ is a basis of $\Omegaoneg$, (left) connections can be
introduced as
\begin{equation}\label{eq:left.conn.def}
  \nabla_{\d_a}\omega^b = \Gamma_{ac}^b\omega^c
\end{equation}
for arbitrary $\Gamma_{ac}^b\in\Ttwotheta$. The connection is torsion
free if
\begin{align*}
  0 = \paraa{\nabla_{\d_1}\omega^a}(\d_2)-\paraa{\nabla_{\d_2}\omega^a}(\d_1)
  =\Gamma_{1c}^a\omega^c(\d_2) - \Gamma_{2c}^a\omega^c(\d_1)
  = \Gamma^a_{12}-\Gamma^a_{21}
\end{align*}
for $a=1,2$; i.e. $\Gamma_{ab}^c=\Gamma_{ba}^c$ for $a,b,c=1,2$.

Let $h$ be a left hermitian form on $\Omegaoneg$, write
$h^{ab}=h(\omega^a,\omega^b)$ and assume that there exist
$h_{ab}\in\Ttwotheta$ such that
$h^{ac}h_{cb}=h_{bc}h^{ca}=\delta^a_b\mid$ for $a,b=1,2$. A connection
of the form \eqref{eq:left.conn.def} is compatible with $h$ if
\begin{align*}
  0 = \d_ch(\omega^a,\omega^b)
  -h\paraa{\nabla_{\d_c}\omega^a,\omega^b}-h\paraa{\omega^a,\nabla_{\d_c}\omega^b}
  = \d_ch^{ab}-\Gamma^a_{cp}h^{pb}-\paraa{\Gamma^b_{cp}h^{pa}}^\ast.
\end{align*}
Introducing $\Gamma^{ab}_c = \Gamma^a_{cp}h^{pb}$
the compatibility condition may be written as
\begin{equation}\label{eq:torus.metric.cond}
  \d_ch^{ab} = \Gamma^{ab}_c+\paraa{\Gamma^{ba}_c}^\ast
\end{equation}
and setting $\Gamma^{ab}_c = \frac{1}{2}\d_ch^{ab}+iS^{ab}_c$
condition \eqref{eq:torus.metric.cond} becomes
$(S^{ab}_c)^\ast=S^{ba}_c$. Hence,
\begin{equation}\label{eq:torus.metric.conn}
  \nabla_{\d_a}\omega^b = \Gamma^b_{ac}\omega^c=\Gamma^{bp}_ah_{pc}\omega^c
  = \paraa{\tfrac{1}{2}\d_ah^{bp}+iS^{bp}_a}h_{pc}\omega^c
\end{equation}
defines a connection compatible with $h$ for arbitrary
$(S^{ab}_c)^\ast=S^{ba}_c$; in particular, choosing $S^{ab}_c=0$ shows
that compatible connections always exist. Let us slightly rewrite the
first term in \eqref{eq:torus.metric.conn} to obtain (using $\d_ah^{bp}h_{pc}=\d_a\delta^b_c\mid=0$)
\begin{equation}
  \nabla_{\d_a}\omega^b = 
  \paraa{-\tfrac{1}{2}h^{bp}\d_ah_{pc}+iS^{bp}_a h_{pc}}\omega^c.
\end{equation}
Furthermore, setting
\begin{align*}
  S^{ab}_c = \tfrac{i}{2}h^{ap}(\d_qh_{pc})h^{qb}
  -\tfrac{i}{2}h^{ap}(\d_ph_{cq})h^{qb} + T^{ab}_c
\end{align*}
gives
\begin{align*}
  (S^{ab}_c)^\ast = -\tfrac{i}{2}h^{bq}(\d_qh_{cp})h^{pa}
  +\tfrac{i}{2}h^{bq}(\d_ph_{qc})h^{pa} + (T^{ab}_c)^\ast
  =S^{ba}_c-T^{ba}_c+(T^{ab}_c)^\ast
\end{align*}
implying that $(S^{ab}_c)^\ast=S^{ba}_c$ is equivalent to $(T^{ab}_c)^\ast=T^{ba}_c$, giving
\begin{align*}
  \Gamma^c_{ab} = -\tfrac{1}{2}h^{cp}\d_ah_{pb}
  -\tfrac{1}{2}h^{cp}\d_bh_{pa}
  +\tfrac{1}{2}h^{cp}\d_ph_{ab}+iT^{cp}_ah_{pb}.
\end{align*}
Thus, the above Christoffel symbols define a connection compatible
with $h$ for arbitrary $T^{ab}_c\in\Ttwotheta$ such that
$(T^{ab}_c)^\ast=T^{ba}_c$. Demanding that the connection is torsion
free amounts to requiring that $\Gamma_{ab}^c=\Gamma_{ba}^c$ which is
equivalent to
\begin{equation}\label{eq:torsion.free.T}
  \tfrac{1}{2}\d_q(h_{ab}-h_{ba}) = ih_{qc}T^{cp}_b h_{pa}-ih_{qc}T^{cp}_a h_{pb}
\end{equation}
Let us split the components of (the inverse of) the hermitian form in
its hermitian and antihermitian part:
\begin{align*}
  h_{ab} = f_{ab} + ig_{ab}
\end{align*}
with $f_{ab}^\ast=f_{ab}$ and $g_{ab}^\ast=g_{ab}$. The condition
$h_{ab}^\ast=h_{ba}$ implies that
\begin{align*}
  f_{ab}=f_{ba}\qand
  g_{ab}=-g_{ba},
\end{align*}
and \eqref{eq:torsion.free.T} becomes
\begin{equation}\label{eq:torsion.free.g.T}
  \d_qg_{ab} = h_{qc}T^{cp}_b h_{pa}-h_{qc}T^{cp}_a h_{pb}
\end{equation}
In fact, since $a\in{1,2}$ and $g_{ab}=-g_{ba}$ there are only two
independent equations:
\begin{align}
  &\d_1g_{12} = h_{1c}T^{cp}_2h_{p1}-h_{1c}T^{cp}_1h_{p2}\label{eq:torsionfree.g.1}\\
  &\d_2g_{12} = h_{2c}T^{cp}_2h_{p1}-h_{2c}T^{cp}_1h_{p2}.\label{eq:torsionfree.g.2}
\end{align}
We note that if $\d_1g_{12}=\d_2g_{12}$=0 then one can solve
\eqref{eq:torsion.free.g.T} by setting $T^{ab}_c=0$. For instance, if
the hermitian form is diagonal (implying that the inverse is also
diagonal) then $\d_qg_{ab}=0$ since $g_{ab}=0$ for $a\neq b$. Thus,
for
\begin{align*}
  (h^{ab}) =
  \begin{pmatrix}
    h_1 & 0 \\ 0 & h_2
  \end{pmatrix}\implies
  (h_{ab}) =
  \begin{pmatrix}
    h_1^{-1} & 0 \\ 0 & h_2^{-1}
  \end{pmatrix}
\end{align*}
a torsion free connection on $\Omegaoneg$ compatible with $h$ is given by
\begin{align*}
  &\nabla_{\d_1}\omega = -\tfrac{1}{2}h_1(\d_1h_1^{-1})\omega-\tfrac{1}{2}h_1(\d_2h_1^{-1})\eta\\
  &\nabla_{\d_1}\eta = \tfrac{1}{2}h_2(\d_2h_1^{-1})\omega-\tfrac{1}{2}h_2(\d_1h_2^{-1})\eta\\
  &\nabla_{\d_2}\omega = -\tfrac{1}{2}h_1(\d_2h_1^{-1})\omega+\tfrac{1}{2}h_1(\d_1h_2^{-1})\eta\\
  &\nabla_{\d_2}\eta = -\tfrac{1}{2}h_2(\d_1h_2^{-1})\omega-\tfrac{1}{2}h_2(\d_2h_2^{-1})\eta.
\end{align*}

\noindent
Let us consider another example where we assume that $h$ is purely off-diagonal, i.e.
\begin{align*}
  (h^{ab}) =
  \begin{pmatrix}
    0 & \hh \\
    \hh^\ast & 0
  \end{pmatrix}\implies
               (h_{ab}) = 
               \begin{pmatrix}
                 0 & (\hhi)^\ast\\
                 \hhi & 0
               \end{pmatrix}=                       
                        \begin{pmatrix}
                          0 & f+ig\\
                          f-ig & 0
                        \end{pmatrix}.
\end{align*}
Setting $T^{11}_a=T^{22}_a=0$ for $a=1,2$, \eqref{eq:torsionfree.g.1}
and \eqref{eq:torsionfree.g.2} become (with $g_{12}=g$)
\begin{align*}
  &\d_1g = -h_{12}T^{21}_2h_{12} = -(\hhi)^\ast T^{21}_1(\hhi)^\ast\\
  &\d_2g = h_{21}T^{12}_2h_{21} = \hhi T^{12}_2\hhi
\end{align*}
which can be solved by setting
\begin{align*}
  &T^{21}_1 = -\hh^\ast(\d_1g)\hh^\ast \qquad
    T^{12}_1=(T^{12}_1)^\ast = -\hh(\d_1g)\hh\\
  &T^{12}_2 = \hh(\d_2g)\hh\qquad
    T^{21}_2 = (T^{12}_2)^\ast = \hh^\ast(\d_2g)\hh^\ast,
\end{align*}
giving
\begin{alignat*}{3}
  &\Gamma^1_{11} = -\hh\d_1f
  &\quad &\Gamma^1_{22}= 0 
  &\quad&\Gamma^1_{12}=\Gamma^1_{21}=i\hh\d_2g\\
  &\Gamma^2_{22} = -\hh^\ast\d_2f 
  & &\Gamma^2_{11} = 0 
  & & \Gamma^2_{12}=\Gamma^2_{21}=-i\hh^\ast\d_1g 
\end{alignat*}
and
\begin{alignat*}{2}
  &\nabla_{\d_1}\omega = -\hh(\d_1f)\omega + i\hh(\d_2g)\eta &\qquad
  &\nabla_{\d_1}\eta = -i\hh^\ast (\d_1g)\eta\\
  &\nabla_{\d_2}\omega = i\hh(\d_2g)\omega &
  &\nabla_{\d_2}\eta = -i\hh^\ast(\d_1g)\omega - \hh^\ast(\d_2f)\eta.
\end{alignat*}

\subsection{Bimodule connections}

Let us now construct bimodule connections on $\Omegaoneg$. To this
end, let $\nabla$ be a left connection on $\Omegaoneg$ and write
\begin{align*}
  \nabla_{\d_a}\omega^c = \Gamma^c_{ab}\omega^b.
\end{align*}
If $\nabla$ is also a right connection, then it has to satisfy
\begin{align*}
  \nabla_{\d_a}(\omega^c f) = \paraa{\nabla_{\d_a}\omega^c}f + \omega^c\d_af.
\end{align*}
for all $f\in\Ttwotheta$. Since $[\omega^c,f]=0$ and $\nabla$ is a left connection, this implies that
\begin{align*}
  f\nabla_{\d_a}\omega^c + (\d_af)\omega^c = \paraa{\nabla_{\d_a}\omega^c}f + \omega^c\d_af\equivalent
  [f,\nabla_{\d_a}\omega^c] = 0\equivalent
  [f,\Gamma^c_{ab}] = 0
\end{align*}
for $f\in\Ttwotheta$ and $a,b,c\in\{1,2\}$. Hence,
$\Gamma^c_{ab}\in Z(\Ttwotheta)$ implying that there exist
$\gamma^c_{ab}\in\complex$ such that
$\Gamma^c_{ab} = \gamma^c_{ab}\mid$. Moreover, in order for a left
connection to be a $\ast$-connection, one needs
\begin{align*}
  0 &= \paraa{\nabla_{\d_a}(f\omega^c)}^\ast-\nabla_{\d_a}(\omega^cf^\ast)
      =\paraa{\nabla_{\d_a}(f\omega^c)}^\ast-\nabla_{\d_a}(f^\ast\omega^c)\\
    &=\paraa{f\nabla_{\d_a}\omega^c+(\d_af)\omega^c}^\ast-f^\ast\nabla_{\d_a}\omega^c-(\d_af^\ast)\omega^c\\
    &=(\nabla_{\d_a}\omega^c)^\ast f^\ast + \omega^c\d_af^\ast-f^\ast\nabla_{\d_a}\omega^c-(\d_af^\ast)\omega^c\\
    &= (\nabla_{\d_a}\omega^c)^\ast f^\ast-f^\ast\nabla_{\d_a}\omega^c
    = \paraa{(\Gamma^c_{ab})^\ast f^\ast -f^\ast\Gamma^{c}_{ab}}\omega^b.
\end{align*}
For $f=\mid$ one immediately obtains
$(\Gamma^c_{ab})^\ast=\Gamma^c_{ab}$, and using that in the above
equation gives $[\Gamma^c_{ab},f^\ast]=0$ for all $f\in\Ttwotheta$,
i.e. $\Gamma^c_{ab}\in Z(\Ttwotheta)$. Hence, a left connection
$\nabla$ is a bimodule connection if $\Gamma^c_{ab}\in Z(\Ttwotheta)$
and, moreover, $\nabla$ is a left $\ast$-connection if
$(\Gamma^c_{ab})^\ast=\Gamma_{ab}^c$.

Clearly, if the hermitian form is constant, i.e. $h^{ab}\sim\mid$ for
$a,b\in\{1,2\}$, then $\Gamma^c_{ab}=0$ gives a torsion free
$\ast$-connection compatible with $h$ on the bimodule
$\Omegaoneg$. This is nothing but the canonical connection on a free
module given by $\nabla_{\d_a}f\omega^c=(\d_af)\omega^c$. However,
there are clearly other solutions when $\d_ch^{ab}=0$. For instance if
\begin{align*}
  (h^{ab}) =
  \begin{pmatrix}
    0 & i\lambda\mid \\
    -i\lambda\mid & 0
  \end{pmatrix}                   
\end{align*}
for $\lambda\in\reals$, then
\begin{align*}
  &\nabla_{\d_1}\eta = \gamma_1\omega \qquad \nabla_{\d_2}\eta = 0\\
  &\nabla_{\d_2}\omega = \gamma_2\eta \qquad \nabla_{\d_1}\omega = 0
\end{align*}
with $\gamma_1,\gamma_2\in\reals$ defines a torsion free
$\ast$-connection compatible with $h$.

Now, assume that a $\ast$-bimodule form $g$ is given on
$\Omegaoneg$ by
\begin{align*}
  g(f_a\omega^a,\omega^b\tilde{f}_b) = f_ag^{ab}\tilde{f}_b
\end{align*}
with $(g^{ab})^\ast=g^{ba}$ and, furthermore, assume that $g$ is
diagonal, i.e.
\begin{align*}
  (g^{ab}) =
  \begin{pmatrix}
    g_1 & 0 \\ 0 & g_2
  \end{pmatrix}.
\end{align*}
Compatibility with $g$ amounts to 
\begin{align*}
  \d_cg^{ab} = \Gamma^a_{cp}g^{pb} + g^{ap}\Gamma^b_{cp},
\end{align*}
and since $g$ is diagonal one obtains the equations
\begin{alignat*}{2}
  &\d_1g_1 = 2\Gamma^1_{11}g_1 &\qquad
  &\d_2g_1 = 2\Gamma^1_{21}g_1\\
  &\d_1g_2 = 2\Gamma^2_{12}g_2 &
  &\d_2g_2 = 2\Gamma^2_{22}g_2\\
  &\Gamma^2_{11}g_1 = -\Gamma^1_{12}g_2 &
  &\Gamma^2_{21} g_1 = -\Gamma^1_{22}g_2.
\end{alignat*}
We note that $g_1$ and $g_2$ are necessarily proportional. For
instance, choosing $g_1=U^kV^l$ and $g_2=zg_1=zU^kV^l$ for
$z\in\complex$ (and $z\neq 0$), one obtains a torsion free bimodule
connection compatible with $g$ by setting
\begin{alignat*}{2}
  &\Gamma^1_{11} = \tfrac{ik}{2} &\qquad
  &\Gamma^1_{12}=\Gamma^1_{21} = \tfrac{il}{2}\\
  &\Gamma^2_{22} = \tfrac{il}{2} &
  &\Gamma^2_{12}=\Gamma^2_{21}=\tfrac{ik}{2}\\
  &\Gamma^1_{22} = -\tfrac{ik}{2z}&
  &\Gamma^2_{11} = -\tfrac{ilz}{2}
\end{alignat*}
giving
\begin{alignat*}{2}
  &\nabla_{\d_1}\omega = i\tfrac{k}{2}\omega + i\tfrac{l}{2}\eta&\qquad
  &\nabla_{\d_1}\eta =  -i\tfrac{lz}{2}\omega + i\tfrac{k}{2}\eta  \\
  &\nabla_{\d_2}\omega = i\tfrac{l}{2}\omega - i\tfrac{k}{2z}\eta  &
  &\nabla_{\d_2}\eta = i\tfrac{k}{2}\omega + i\tfrac{l}{2}\eta.
\end{alignat*}
Note that $\nabla$ is not a $\ast$-connection since not all $\Gamma^c_{ab}$ are real.

\section{The geometry of Kronecker algebras}\label{sec:kronecker.alg}

Let us turn to our main object of study in this paper; namely, the
path algebra of the (generalized) Kronecker quiver with $N$ arrows:
\begin{displaymath}
  \begin{tikzcd}
    1
    \arrow[r, draw=none, "\raisebox{+1.5ex}{\vdots}" description]
    \arrow[r, bend left,        "\alpha_1"]
    \arrow[r, bend right, swap, "\alpha_N"]
    &
    2
  \end{tikzcd}
\end{displaymath}
Denoting the leftmost node by $e$, and consequently the rightmost node by $\mid-e$, we let
$\KN$ be the unital $\complex$-algebra generated by
$e,\alpha_1,\ldots,\alpha_N$ satisfying
\begin{align}\label{eq:kronecker.alg.relations}
  e^2=e\qquad
  e\alpha_k = \alpha_k\qquad
  \alpha_ke = 0\qquad
  \alpha_j\alpha_k = 0
\end{align}
for $j,k\in\{1,\ldots,N\}$. The algebra $\KN$ is finite
dimensional, and every element $a\in\KN$ can be uniquely written as
\begin{align*}
  a = \lambda\mid + \mu e + a^i\alpha_i
\end{align*}
for $\lambda,\mu,a^i\in\complex$,  where summation from $1$ to $N$ over the
repeated index $i$ is assumed. The product of two elements
\begin{align*}
  a_1 = \lambda_1\mid + \mu_1e + a_1^i\alpha_i\quad\text{and}\quad
  a_2 = \lambda_2\mid + \mu_2e + a_2^i\alpha_i
\end{align*}
can then be computed as
\begin{align}
  a_1a_2 = \lambda_1\lambda_2\mid +
  \paraa{\lambda_1\mu_2+\lambda_2\mu_1+\mu_1\mu_2}e+
  \paraa{(\lambda_1+\mu_1)a_2^i+\lambda_2a_1^i}\alpha_i,
\end{align}
giving
\begin{align}\label{eq:a1a2.commutator}
  [a_1,a_2] = \paraa{\mu_1a_2^i-\mu_2a_1^i}\alpha_i.
\end{align}
The subspace $\KaN$, generated by $\alpha_1,\ldots,\alpha_N$, is a
two-sided ideal of $\KN$ and we note that $ea=a$ and $ab=0$ for all
$a,b\in\KaN$.

It is straightforward to introduce a $\ast$-structure on $\KN$.

\begin{proposition}\label{prop:star.algebra.structure}
  For $a=\lambda\mid + \mu e + a^i\alpha_i\in\KN$ set
  \begin{align}
    a^\ast = (\bar{\lambda}+\bar{\mu})\mid-\bar{\mu} e+\bar{a}^i\alpha_i.
  \end{align}
  Then $(a^\ast)^\ast = a$, $(za+b)^\ast=\bar{z}a^\ast + b^\ast$ and
  $(ab)^\ast=b^\ast a^\ast$ for all $z\in\complex$ and $a,b\in\KN$.
\end{proposition}

\begin{proof}
  It is easy to see that $(za+b)^\ast = \bar{z}a^\ast +
  b^\ast$. Furthermore, one checks that
  \begin{align*}
    (a^\ast)^\ast = (\lambda+\mu-\mu)\mid+\mu e+a^i\alpha_i = a.
  \end{align*}
  Now, for
  \begin{align*}
    a_1 = \lambda_1\mid + \mu_1e + a_1^i\alpha_i\quad\text{and}\quad
    a_2 = \lambda_2\mid + \mu_2e + a_2^i\alpha_i
  \end{align*}
  one computes
  \begin{align*}
    (a_1a_2)^\ast = &(\lb_1+\mub_1)(\lb_2+\mub_2)\mid
                      -(\lb_1\mub_2+\lb_2\mub_1+\mub_1\mub_2)e+\paraa{(\lb_1+\mub_1)\bar{a}_2^i+\lb_2\bar{a}_1^i}
  \end{align*}
  and
  \begin{align*}
    a_2^\ast a_1^\ast
    &= \parab{(\lb_2+\mub_2)\mid-\mub_2 e+\bar{a}_2^i\alpha_i}
    \parab{(\lb_1+\mub_1)\mid-\mub_1 e+\bar{a}_1^i\alpha_i}\\
    &=(\lb_1+\mub_1)(\lb_2+\mub_2)\mid
      +\paraa{-(\lb_2+\mub_2)\mub_1-(\lb_1+\mub_1)\mub_2+\mub_1\mub_2}e\\
    &\quad+\paraa{\lb_2\bar{a}_1^i+(\lb_1+\mub_1)\bar{a}_2^i}\alpha_i,
  \end{align*}
  showing that $(a_1a_2)^\ast=a_2^\ast a_1^\ast$ for all
  $a_1,a_2\in\KN$.
\end{proof}

\noindent
Proposition~\ref{prop:star.algebra.structure} implies that $\KN$ is a
$\ast$-algebra with
\begin{align*}
  e^\ast = \mid-e\qand \alpha_i^\ast = \alpha_i.
\end{align*}
Moreover, note that an arbitrary hermitian element $a\in\KN$ can be
written as
\begin{align*}
  a = (\lambda-\tfrac{i}{2}\mu)\mid+i\mu e+a^i\alpha_i
\end{align*}
with $\lambda,\mu,a^i\in\reals$.

In the construction of differential calculi, the center plays an
important role. For $\KN$, the center turns out to be trivial.

\begin{proposition}
  $Z(\KN) = \{\lambda\mid:\lambda\in\complex\}$.
\end{proposition}

\begin{proof}
  Let $c=\lambda\mid + \mu e + a^i\alpha_i$ be in the center of
  $\KN$. Writing out $[c,e] = 0$ gives
  \begin{align*}
    0=[c,e] = ce - ec = (\lambda+\mu e)\mid -(\lambda+\mu e)\mid-a^i\alpha_i
    =-a^i\alpha_i
  \end{align*}
  implying that $a^i=0$ for $i=1,\ldots,N$ and $c=\lambda\mid+\mu
  e$. Requiring that $[c,\alpha_i]=0$ for $i=1,\ldots,N$ gives
  \begin{align*}
    0 = [c,\alpha_i] = c\alpha_i-\alpha_ic
    = (\lambda+\mu)\alpha_i-\lambda\alpha_i = \mu\alpha_i
  \end{align*}
  implying that $\mu=0$ and $c=\lambda\mid$. Clearly, any element of
  the form $\lambda\mid$ is in the center of $\KN$.
\end{proof}

\noindent
Furthermore, the algebra $\KN$ has plenty of invertible elements, as shown by the
following result.

\begin{proposition}\label{prop:invertible}
  An element $a=\lambda\mid+\mu e+a^i\alpha_i\in\KN$ is invertible if
  and only if $\lambda\neq 0$ and $\lambda+\mu\neq 0$. In this case,
  the inverse is given by
  \begin{equation}\label{eq:a.inverse}
    a^{-1} = \frac{1}{\lambda(\lambda+\mu)}\parab{(\lambda+\mu)\mid -\mu e - a^i\alpha_i}.
  \end{equation}
\end{proposition}

\begin{proof}
  If $\lambda\neq 0$ and $\lambda+\mu\neq 0$ it is easy to check that
  \eqref{eq:a.inverse} is the inverse of $a$. Now, assume that
  $a=\lambda\mid+\mu e+a^i\alpha_i$ is invertible. Then there exists
  \begin{equation*}
    b = \gamma_1\mid + \gamma_2e + \gamma^i\alpha_i
  \end{equation*}
  such that $ab=\mid$, which is equivalent to
  \begin{align}
    &\lambda\gamma_1 = 1\label{eq:lambda.gamma1}\\
    &(\lambda+\mu)\gamma_2 = -\gamma_1\mu\label{eq:lambda.gamma2}\\
    &(\lambda+\mu)\gamma^i = -\gamma_1a^i\label{eq:lambda.gammai}
  \end{align}
  for $i=1,\ldots,N$. It follows immediately from
  \eqref{eq:lambda.gamma1} that $\lambda\neq 0$ and $\gamma_1\neq
  0$. Furthermore, if $\lambda+\mu=0$ then \eqref{eq:lambda.gamma2}
  implies that $\mu=0$ (since $\gamma_1\neq 0$) and, consequently,
  that $\lambda=-\mu=0$, which contradicts \eqref{eq:lambda.gamma1}.
  Hence, if $ab=\mid$ then $\lambda\neq 0$ and $\lambda+\mu\neq 0$.
\end{proof} 

\noindent
In the following, we will construct derivation based differential
calculi over $\KN$. As a first step, let us describe the Lie algebra
of derivations as well as a basis consisting of hermitian derivations.

\begin{proposition}\label{prop:derivations}
  A basis of $\Der(\KN)$ is given by $\{\d_k\}_{k=1}^N$ and $\{\d_k^l\}_{k,l=1}^N$ with
  \begin{alignat}{2}
    &\d_k(e) = i\alpha_k &\qquad &\d_k(\alpha_l) = 0\label{eq:di.def}\\
    &\d_k^l(e)=0 &\qquad &\d_{k}^l(\alpha_j) =\delta_j^l\alpha_k,\label{eq:dij.def}
  \end{alignat}
  satisfying
  \begin{align*}
    &[\d_{i}^j,\d_{k}^l] = \delta_{k}^j\d_{i}^l-\delta_{i}^l\d_{k}^j\\
    &[\d_{i}^j,\d_k] = \delta_{k}^j\d_i\\
    &[\d_i,\d_j] = 0.
  \end{align*}
  Moreover, $\d_i,\d_i^j$ are hermitian derivations for $i,j=1,\ldots,N$.
\end{proposition}

\begin{proof}
  Let us start by finding the form of an arbitrary derivation. To this
  end, we make the Ansatz:
  \begin{align*}
    &\d(e) = \lambda\mid + \mu e + a^k\alpha_k\\\
    &\d(\alpha_i) = \lambda_i\mid + \mu_ie+ b_i^k\alpha_k
  \end{align*}
  for $\lambda,\lambda_i,\mu,\mu_i,a^k,b_i^k\in\complex$. These maps
  (extended as derivations to all of $\KN$) are derivations if they
  preserve the relations in \eqref{eq:kronecker.alg.relations}. One computes
  \begin{align*}
    0 &= \d(e^2-e) = e\d e+(\d e)e -\d e\\
      &= e(\lambda\mid + \mu e + a^k\alpha_k)+(\lambda\mid + \mu e + a^k\alpha_k)e-(\lambda\mid + \mu e + a^k\alpha_k)\\
      &= -\lambda\mid  + (2\lambda+\mu)e,
  \end{align*}
  implying that $\lambda=\mu=0$ and $\d e = a^k\alpha_k$. Next,
  \begin{align*}
    0 &= \d(\alpha_ie) = \alpha_i\d e + (\d\alpha_i)e
        = \alpha_i(a^k\alpha_k) + (\lambda_i\mid + \mu_ie + b_i^k\alpha_k)e\\
      &= (\lambda_i+\mu_i)e,
  \end{align*}
  implying that $\mu_i=-\lambda_i$ and $\d\alpha_i = \lambda_i(\mid-e)+b_i^k\alpha_k$. Furthermore,
  \begin{align*}
    0 &= \d(e\alpha_i-\alpha_i) = e\d\alpha_i + (\d e)\alpha_i - \d\alpha_i\\
      &= e\paraa{\lambda_i(\mid-e)+b_i^k\alpha_k} + a^k\alpha_k\alpha_i-\lambda_i(\mid-e)-b_i^k\alpha_k\\
      &= -\lambda_i(\mid-e),
  \end{align*}
  implying that $\lambda_i=0$ and $\d\alpha_i = b_i^k\alpha_k$. Finally,
  \begin{align*}
    0=\d(\alpha_i\alpha_j) = \alpha_i\d\alpha_j + (\d\alpha_i)\alpha_j
    = b_j^k\alpha_i\alpha_k + b_i^k\alpha_k\alpha_j = 0.
  \end{align*}
  Hence, $\d$ is a derivation if and only if there exists $a^k,b_i^k\in\complex$ such that
  \begin{align*}
    \d e = a^k\alpha_k\qand
    \d\alpha_i = b_i^k\alpha_k.
  \end{align*}
  It is now clear that the derivations given in \eqref{eq:di.def} and
  \eqref{eq:dij.def} is a vector space basis of $\Der(\KN)$. Next, let
  us compute the commutators of these derivations:
  \begin{align*}
    &[\d_k,\d_l](e) = \d_k(i\alpha_l) - \d_l(i\alpha_k) = 0\\
    &[\d_k,\d_l](\alpha_m) = \d_k(0) - \d_l(0)=0
  \end{align*}
  gives $[\d_k,\d_l]=0$, and
  \begin{align*}
    [\d_i^j,\d_{k}^l](e) &= \d_i^j(0)-\d_k^l(0) = 0\\
    [\d_i^j,\d_k^l](\alpha_m) &= \delta_m^l\d_i^j(\alpha_k)-\delta_m^j\d_k^l(\alpha_i)
      =\delta_m^l\delta_k^j\alpha_i-\delta_m^j\delta_i^l\alpha_k\\
                         &= \delta_k^j\d_i^l(\alpha_m)-\delta_i^l\d_k^j(\alpha_m),
  \end{align*}
  from which we conclude that $[\d_i^j,\d_k^l]=\delta_k^j\d_i^l-\delta_i^l\d_k^j$. Furthermore,
  \begin{align*}
    &[\d_j^k,\d_l](e) = \d_j^k(i\alpha_l) = i\delta_l^k\alpha_j = \delta^k_l\d_j(e)\\
    &[\d_j^k,\d_l](\alpha_m) = -\delta_m^k\d_l\alpha_j = 0,
  \end{align*}
  giving $[\d_j^k,\d_l] = \delta_k^l\d_j$. Finally, let us show that
  $\d_j$ and $\d_j^k$ are hermitian derivations; namely
  \begin{align*}
    &\d_j(e^\ast)^\ast = \d_j(\mid-e)^\ast=-(i\alpha_j)^\ast = i\alpha_j = \d_j(e)\\
    &\d_j^k(\alpha_m^\ast)^\ast = \d_j^k(\alpha_m)^\ast = \delta^k_m\alpha_j = \d_j^k(\alpha_m)
  \end{align*}
  together with $\d_j(\alpha_k^\ast)^\ast=\d_j^k(e^\ast)^\ast = 0$
  shows that $\d_j$ and $\d_j^k$ are indeed hermitian derivations for
  $j,k=1,\ldots,N$.
\end{proof}

\noindent
From the above result, we note in particular that $\d(a)\in\KaN$ for
all $\d\in\Der(\KN)$ and $a\in\KN$. Proposition~\ref{prop:derivations}
constructs a basis for all derivations of $\KN$, and the next results
determines which derivations are inner in terms of the basis.

\begin{proposition}
  A derivation $\d\in\Der(\KN)$ is inner if and only if there exist
  $a^0,a^1,\ldots,a^N\in\complex$ such that
  \begin{align}\label{eq:inner.derivation}
    \d = a^i\d_i + a^0\paraa{\d_1^1+ \d_2^2+ \cdots + \d_N^N}.
  \end{align}
\end{proposition}

\begin{proof}
  An inner derivation is by definition of the form
  $\d(\cdot)=[\cdot,a]$ for some $a\in\KN$. Setting
  $a=\lambda\mid+\mu e+a^i\alpha_i$ one obtains
  \begin{align*}
    &\d e = [e,a] = a^i\alpha_i\\
    &\d\alpha_k = -\mu\alpha_k
  \end{align*}
  implying that
  \begin{align*}
    \d = a^i\d_i - \mu\paraa{\d^1_1+\d^2_2+\cdots+\d_N^N},
  \end{align*}
  which proves that every inner derivation is of the form as in
  \eqref{eq:inner.derivation}.
\end{proof}

\noindent
Let us denote by $\ginn$ the Lie algebra of inner derivations; that
is, $\ginn$ is a Lie algebra of dimension $N+1$ with a basis given by
$\{\dh,\d_1,\ldots,\d_N\}$, with $\dh=\d_1^1+\cdots+\d_N^N$, satisfying
\begin{align*}
  [\d_i,\d_j] = 0\qand
  [\dh,\d_i] = \d_i
\end{align*}
for $i=1,\ldots,N$ and
\begin{align*}
  \dh(a) = [e,a]\qquad \d_k(a) = -i[\alpha_k,a]
\end{align*}
for $a\in\KN$.

\subsection{Differential calculi}

Given a Lie algebra $\g\subseteq\Der(\KN)$ (which we shall always
assume to be closed under $\d\to\d^\ast$), we construct the restricted
calculus $\Omegag{}$. Recall that $\Omegaoneg$ is generated by $da$
for $a\in\KN$ and, since $\KN$ is finite dimensional, $\Omegaoneg$ is
generated by $d\alpha_0\equiv ide, d\alpha_1,\ldots,
d\alpha_N$. However, depending on the Lie algebra $\g$, these 1-forms
are in general not linearly independent. Nevertheless, one can express
the bimodule structure of $\Omegaoneg$ independent of the choice of
$\g$.

\begin{proposition}\label{prop:Omega1.bimodule.structure}
  For any $\g\subseteq\Der(\KN)$ the bimodule structure of $\Omegag{1}$ is given by
  \begin{align*}
    &ed\alpha_I = d\alpha_I\qquad
    \alpha_id\alpha_I = 0\\
    &(d\alpha_I)e = 0\qquad (d\alpha_I)\alpha_i = 0,
  \end{align*}
  for $i=1,\ldots,N$ and $I=0,1,\ldots,N$; furthermore, $(d\alpha_I)^\ast=d\alpha_I$ for $I=0,1\ldots,N$.
\end{proposition}

\begin{proof}
  Since $\d a\in\KaN$ for all $\d\in\Der(\KN)$ and $a\in\KN$ one obtains
  \begin{align*}
    &ed\alpha_I(\d) = e\d(\alpha_I) = \d(\alpha_I)=d\alpha_I(\d)\\
    &\alpha_id\alpha_I(\d) = \alpha_i\d(\alpha_I) = 0
  \end{align*}
  and in the same way one obtains the right bimodule relations.
  Next, one computes
  \begin{align*}
    &(d\alpha_k)^\ast(\d) = \paraa{d\alpha_k(\d^\ast)}^\ast
      = \paraa{\d^\ast\alpha_k}^\ast = \d(\alpha_k^\ast)
      =\d\alpha_k = d\alpha_k(\d)\\
    &(d\alpha_0)^\ast(\d) = \paraa{ide(\d^\ast)}^\ast
      = -i\paraa{\d^\ast e}^\ast = -i\d(e^\ast) = -i\d(\mid-e)
      =i\d e = d\alpha_0(\d)
  \end{align*}
  for $\d\in\g$, showing that $(d\alpha_I)^\ast=d\alpha_I$.
\end{proof}

\noindent
The above result implies that $\Omegaoneg$ is generated by
$\{d\alpha_I\}_{I=0}^N$ as a complex vector space; i.e. every element
$\omega\in\Omegaoneg$ can be written as $\omega=\lambda^Id\alpha_I$
for $\lambda^I\in\complex$. Moreover,
Proposition~\ref{prop:Omega1.bimodule.structure} implies that
$\Omegaoneg$ is not a free module since $\alpha_k\omega=0$ for any
$\omega\in\Omegaoneg$.

For $\KN$, it turns out that there are no higher order forms, as the
next result shows.

\begin{proposition}\label{prop:Omegak.zero}
  If $\g\subseteq\Der(\KN)$ then $\Omegag{k}=0$ for $k\geq 2$.
\end{proposition}

\begin{proof}
  The module $\Omegag{k}$ is generated by elements of the form
  $da_1da_2\cdots da_k$ for $a_k\in\KN$, and
  \begin{align*}
    da_1da_2\cdots da_k(\d_1,\ldots,\d_k)
    =\sum_{\sigma\in S_k}\sgn(\sigma)(\d_{\sigma(1)}a_1)(\d_{\sigma(2)}a_2)\cdots(\d_{\sigma(k)}a_k).
  \end{align*}
  For $k\geq 2$ the above expression is zero since $\d a\in\KaN$ for
  all $\d\in\Der(\KN)$ and $a\in\KN$.
\end{proof}

\noindent
Recall that the $k$'th cohomology $H^k(\Omegag{})$ is defined  as the
closed $k$-forms modulo the exact $k$-forms; that is,
\begin{align*}
  H^k(\Omegag{}) = \im(d_{k-1})\slash \ker(d_k).
\end{align*}
Since $\Omegag{k}=0$ for $k\geq 2$, by
Proposition~\ref{prop:Omegak.zero}, the only remaining cohomology is
$H^1(\Omegag{})$.

\begin{proposition}
  If $\g\subseteq\Der(\KN)$ then $H^1(\Omegag{})=0$.
\end{proposition}

\begin{proof}
  Since $\Omegag{k}=0$ for $k\geq 2$ (by
  Proposition~\ref{prop:Omegak.zero}) every element of $\Omegag{1}$ is
  closed. Moreover, since every element of $\Omegag{1}$ can be written
  as $\lambda^Id\alpha_I$ for $\lambda^I\in\complex$  it follows immediately that
  every element of $\Omegag{1}$ is exact, since
  \begin{equation*}
    \lambda^0de + \lambda^1d\alpha_1+\cdots+\lambda^Nd\alpha_N
    =d\paraa{\lambda^0e + \lambda^1\alpha_1+\cdots+\lambda^N\alpha_N}.\qedhere
  \end{equation*}
\end{proof}

\noindent
Let us now consider the structure of $\Omegaoneg$ in more detail.
Given $a\in\A$ it follows from
Proposition~\ref{prop:Omega1.bimodule.structure} that
\begin{align*}
  \angles{da} = \{\lambda da:\lambda\in\complex\}
\end{align*}
is a sub-bimodule of $\Omegag{1}$ with
\begin{equation}\label{eq:da.module.relations}
  \begin{split}
    &e da = da\qquad da\cdot e = 0\\
    &\alpha_ida = 0 \qquad da\cdot \alpha_i =0,
  \end{split}
\end{equation}
and due to the above bimodule relations one finds that when $da\neq 0$
\begin{align*}
  \angles{da}\simeq\{\lambda\alpha_k:\lambda\in\complex\}
  \equiv M_{\alpha_k}
\end{align*}
as a $\KN$-bimodule for $k=1,\ldots,N$.

\begin{proposition}\label{prop:Malphak.projective.bimodule}
  $M_{\alpha_k}=\{\lambda\alpha_k:\lambda\in\complex\}$ is a
  projective $\KN$-bimodule.
\end{proposition}

\begin{proof}
  Let $\KN^e=\KN\otimesC\KN^{op}$ denote the enveloping algebra of
  $\KN$, and let $\hat{e}\in\KN^e$ denote the idempotent
  $\hat{e}=e\otimesC\overline{(\mid-e)}$, where $\bar{a}$ denotes
  $a\in\KN$ as an element in $\KN^{op}$;
  i.e. $\bar{a}\cdot\bar{b}=\overline{ba}$. Then $\KN^e\hat{e}$ is a
  projective left $\KN^e$-module. Let us now show that $M_{\alpha_k}$
  is isomorphic to $\KN^e\hat{e}$ as a left $\KN^e$-module. By
  definition, an arbitrary element of $\KN^e\hat{e}$ can be written as
  sums of elements of the form
  \begin{align*}
    (a\otimesC\bar{b})(e\otimesC\overline{(\mid-e)})
    =ae\otimesC\overline{(\mid-e)b}
  \end{align*}
  and by using the algebra relations one concludes that
  \begin{align*}
    \KN^e\hat{e} = \{\lambda\paraa{e\otimesC\overline{(\mid-e)}}:\lambda\in\complex\}.
  \end{align*}
  Defining $\phi:\KN^e\hat{e}\to M_{\alpha_k}$ as
  \begin{align*}
    \phi\paraa{\lambda\paraa{e\otimesC\overline{(\mid-e)}}} = \lambda\alpha_k
  \end{align*}
  it is clear that $\phi$ is a vector space isomorphism. Furthermore, for
  \begin{align*}
    a_1 = \lambda_1\mid+\mu_1e + a_1^i\alpha_i\qand
    a_2 = \lambda_2\mid+\mu_2e + a_2^i\alpha_i
  \end{align*}
  one finds that
  \begin{align*}
    &\phi\paraa{(a_1\otimesC \bar{a}_2)e\otimesC\overline{(\mid-e)}}
      =\phi\paraa{a_1e\otimesC\overline{((\mid-e)a_2)}}
      =(\lambda_1+\mu_1)\lambda_2\alpha_k\\
    &(a_1\otimesC\bar{a}_2)\phi\paraa{e\otimesC\overline{(\mid-e)}}
      = a_1\alpha_k a_2 = (\lambda_1+\mu_1)\alpha_k a_2 = (\lambda_1+\mu_1)\lambda_2\alpha_k
  \end{align*}
  showing that $\phi$ is an isomorphism of left
  $\KN^e$-modules. Hence, $M_k$ is projective as a left $\KN^e$-module
  or, equivalently, as a $\KN$-bimodule.
\end{proof}

\noindent
In particular, Proposition~\ref{prop:Malphak.projective.bimodule}
implies that $\angles{da}$ is projective as a
$\KN$-bimodule (and consequently projective as both a left and right
$\KN$-module).

As we have seen, for any choice of Lie algebra $\g\subseteq\Der(\KN)$,
the module $\Omegaoneg$ is a finite dimensional vector space spanned by
$d\alpha_0,d\alpha_1,\ldots,d\alpha_N$; however, the dimension of $\Omegaoneg$
clearly depends on the choice of $\g$. Let
$n=\dim_{\complex}(\Omegag{1})$ and let $\{dx_a\}_{a=1}^n$ denote a (vector
space) basis of $\Omegag{1}$. Hence, for each $\omega\in\Omegag{1}$
there exist unique $\lambda^1,\ldots,\lambda^n\in\complex$ such that
$\omega=\lambda^adx_a$.

\begin{proposition}\label{prop:Omegag.direct.sum}
  Let $\g\subseteq\Der(\KN)$ and assume that $\{dx_a\}_{a=1}^n$ is a
  (vector space) basis of $\Omegaoneg$. Then
  \begin{align*}
    \Omegaoneg\simeq \bigoplus_{a=1}^n\angles{dx_a}
  \end{align*}
  as a $\KN$-bimodule.
\end{proposition}

\begin{proof}
  Let $M=\oplus_{a=1}^n\angles{dx_a}$ and defining $\phi:\Omegaoneg\to M$
  as
  \begin{align*}
    \phi(\lambda^adx_a) = \lambda^1dx_1\oplus\cdots\oplus\lambda^ndx_n
  \end{align*}
  it is clear that $\phi$ is an isomorphism of vector spaces. To prove
  that $\phi$ is a left module homomorphism one checks that
  \begin{align*}
    &\phi(e \lambda^adx_a)-e\phi(\lambda^adx_a)
    = \phi(\lambda^adx_a)-e\phi(\lambda^adx_a)\\
      &\qquad\qquad=(\mid-e)\paraa{\lambda^1dx_1\oplus\cdots\oplus\lambda^ndx_n}=0\\
    &\phi(\alpha_i\lambda^adx_a)-\alpha_i\phi(\lambda^adx_a)
      = -\alpha_i\paraa{\lambda^1dx_1\oplus\cdots\oplus\lambda^ndx_n}=0,
  \end{align*}
  and to show that $\phi$ is a right module homomorphism, one computes
  \begin{align*}
    &\phi(\lambda^adx_a\cdot e)-\phi(\lambda^adx_a)e
    =-\paraa{\lambda^1dx_1\oplus\cdots\oplus\lambda^ndx_n}e=0\\
    &\phi(\lambda^adx_a\cdot \alpha_i)-\phi(\lambda^adx_a)\alpha_i
    =-\paraa{\lambda^1dx_1\oplus\cdots\oplus\lambda^ndx_n}\alpha_i=0.
  \end{align*}
  We conclude that $\phi:\Omegaoneg\to M$ is a bimodule isomorphism.  
\end{proof}

\noindent
The following is then an immediate consequence of the above result.

\begin{corollary}
  $\Omegaoneg$ is a projective $\KN$-bimodule.
\end{corollary}

\noindent
Recall that a differential calculus is called connected if $da=0$
implies that $a=\lambda\mid$ for some $\lambda\in\complex$. For $\Omegaoneg$,
one proves the following criterion for connectedness.

\begin{proposition}\label{prop:connected.vspace.basis}
  $\Omegaoneg$ is connected if and only if $d\alpha_0,d\alpha_1,\ldots,d\alpha_N$ is
  a vector space basis of $\Omegaoneg$.
\end{proposition}

\begin{proof}
  Let $a=\lambda\mid + \mu e + a^i\alpha_i$ be an arbitrary element of
  $\KN$. $\Omegaoneg$ being connected is equivalent to
  \begin{align*}
    &da = 0\implies a=\lambda\mid \equivalent\\
    &-i\mu d\alpha_0 + a^id\alpha_i = 0 \implies \mu=a^i=0
  \end{align*}
  which is equivalent to
  $d\alpha_0,d\alpha_1,\ldots,d\alpha_N$ being a vector space basis of
  $\Omegaoneg$.
\end{proof}

\noindent
Let us consider the case when $\g=\Der(\KN)$ in which case we write
$\Omegaoneg=\Omega^1_{\Der}$. In that case it is easy to see
that $\Omega_{\Der}^1$ is connected. Namely, for
$a=\lambda\mid + \mu e + a^i\alpha_i$, assuming $da=0$ gives
\begin{align*}
  &da(\d_k)=0 \implies \mu\d_ke = 0\implies i\mu\alpha_k = 0\implies \mu = 0\\
  &da(\d^k_l) = 0\implies a^i\d^k_l\alpha_i = 0\implies a^k\alpha_l = 0\implies a^k=0
\end{align*}
implying that $a=\lambda\mid$, giving that $\{d\alpha_I\}_{I=0}^N$ is
a vector space basis of $\Omega_{\Der}^1$ via
Proposition~\ref{prop:connected.vspace.basis}. Furthermore,
Proposition~\ref{prop:Omegag.direct.sum} implies that
\begin{align*}
  \Omega^1_{\Der} \simeq \bigoplus_{I=0}^N\angles{d\alpha_I}
\end{align*}
as a $\KN$-bimodule. 

\subsection{Traces and integration}

A trace on a $\ast$-algebra $\A$ is a $\complex$-linear functional
$\tau:\A\to\complex$ such that $\tau(a^\ast)=\overline{\tau(a)}$ and
$\tau(ab)=\tau(ba)$ for all $a,b\in\A$. If $\tau(\mid)=1$ then the
trace is said to be normalized. Moreover, if $\tau(a^\ast a)\geq 0$
for all $a\in\A$ then $\tau$ is said to be a positive trace. Note
that, if the algebra is finite dimensional, a trace is determined by
its values on the basis elements.

\begin{proposition}\label{prop:traces.KN}
  A $\complex$-linear functional $\tau:\KN\to\complex$ is a trace on $\KN$ if
  and only if $\tau(\alpha_i)=0$ for $i=1,\ldots,N$ and there exist
  $\tau_0,\tau_1\in\reals$ such that
  \begin{align*}
    \tau(\mid) = \tau_0\qand
    \tau(e) = \tfrac{1}{2}\tau_0+i\tau_1.
  \end{align*}
\end{proposition}

\begin{proof}
  Let $\tau$ be a $\complex$-linear functional on $\KN$, let
  $a_1,a_2\in\KN$ be arbitrary elements and write
  \begin{align*}
    &a_1 = \lambda_1\mid+\mu_1e+a_1^i\alpha_i\\
    &a_2 = \lambda_2\mid+\mu_2e+a_2^i\alpha_i.
  \end{align*}
  One finds that
  \begin{align*}
    \tau(a_1a_2)-\tau(a_2a_1)= (\mu_1a_2^i-\lambda_1a_1^i)\tau(\alpha_i),
  \end{align*}
  and requiring this to be zero for all $a_1,a_2\in\KN$ is equivalent
  to $\tau(\alpha_i)=0$ for $i=1,\ldots,N$. Furthermore, for
  $a=\lambda\mid+\mu e+a^i\alpha_i$, $\tau(a^\ast)=\overline{\tau(a)}$ for all $a\in\KN$
  is equivalent to
  \begin{align*}
    \bar{\lambda}\paraa{\tau(\mid)-\overline{\tau(\mid)}}
    +\bar{\mu}\paraa{\tau(\mid)-\tau(e)-\overline{\tau(e)}}=0
  \end{align*}
  for all $\lambda,\mu\in\complex$, which is equivalent to
  \begin{align*}
    \tau(\mid)=\tau_0\qand \tau(e)=\tfrac{1}{2}\tau_0+i\tau_1 
  \end{align*}
  for some $\tau_0,\tau_1\in\reals$.
\end{proof}

\noindent 
It follows that there is a 1-parameter family of normalized traces on
$\KN$, given by
\begin{align*}
  \tau(\mid) = 1\qquad \tau(e) = \tfrac{1}{2}+i\tau_1\qquad \tau(\alpha_i)=0
\end{align*}
for $\tau_1\in\reals$.

The next result shows that there are no non-trivial positive traces on
$\KN$.

\begin{proposition}
  If $\tau$ is a trace on $\KN$ such that $\tau(a^\ast a)\geq 0$ for
  all $a\in\KN$ then $\tau(a)=0$ for all $a\in\KN$.
\end{proposition}

\begin{proof}
  Let $\tau$ be a trace on $\KN$. From Proposition~\ref{prop:traces.KN} it follows that
  \begin{align*}
    \tau(\mid) = \tau_0\qquad
    \tau(e) = \tfrac{1}{2}\tau_0+i\tau_1\qquad
    \tau(\alpha_i) = 0\quad i=1,\ldots,N
  \end{align*}
  for some $\tau_0,\tau_1\in\complex$. For
  $a=\lambda\mid+\mu e+a^i\alpha_i$ one obtains
  \begin{align*}
    \tau(a^\ast a)
    &= (|\lambda|^2+\lambda\bar{\mu})\tau_0
    +(\mu\bar{\lambda}-\bar{\mu}\lambda)\paraa{\tfrac{1}{2}\tau_0+i\tau_1}\\
    &= \paraa{|\lambda|^2+\Re(\lambda\bar{\mu})}\tau_0 + 2\Im(\lambda\bar{\mu})\tau_1. 
  \end{align*}
  If $\tau_0\neq 0$ one can always find $\lambda,\mu\in\reals$ such that
  \begin{align*}
    \tau(a^\ast a) = \lambda(\lambda+\mu)\tau_0<0,
  \end{align*}
  and if $\tau_1\neq 0$ one can always find $\lambda',\mu\in\reals$ such that
  for $\lambda=i\lambda'$
  \begin{align*}
    \tau(a^\ast a) = \lambda'^2\tau_0 + 2\mu\lambda'\tau_1
    =\lambda'(\lambda'\tau_0+2\mu\tau_1)<0.
  \end{align*}
  Hence, if $\tau(a^\ast a)\geq 0$ for all $a\in\KN$ then
  $\tau_0=\tau_1=0$, implying that $\tau(a)=0$ for all $a\in\KN$.
\end{proof}

\noindent
Now, we would like to use traces to integrate elements of $\Omegaoneg$
and we will proceed in analogy with a $1$-dimensional interval
$[a,b]\subseteq\reals$. For $f\in C^\infty([a,b])$ one has
\begin{align*}
  \int_a^b df = \int_a^b\frac{df}{dx}dx = f(b)-f(a),
\end{align*}
and $\tilde{\tau}(f)=f(b)-f(a)$ is clearly a trace on the
(commutative) algebra $C^\infty([a,b])$. Thus, setting
\begin{align*}
  \int da = \tau(a)
\end{align*}
for $a\in\KN$ the map $\int:\Omegaoneg\to\complex$ is well defined if
$\tau(a)=0$ whenever $da=0$. Assuming that $\Omegaoneg$ is connected
(giving $da=0\Rightarrow a=\lambda\mid$) it is necessary that
$\tau(\mid)=\tau_0=0$, which is consistent with $f(b)-f(a)=0$ for any
constant function $f$. It follows that
\begin{align*}
  \int d\alpha_0 = i\tau(e) = i(i\tau_1) = -\tau_1\qqand \int d\alpha_k = 0
\end{align*}
for $k=1,\ldots,N$.

\subsection{Connections and hermitian forms}

Let us now consider connections on $\Omegaoneg$; i.e bilinear maps
$\nabla:\g\times\Omegaoneg\to\Omegaoneg$ satisfying a Leibniz rule,
which for left connections reads
\begin{align*}
  \nabla_\d (a\omega) = a\nabla_{\d}\omega + (\d a)\omega
\end{align*}
for $a\in\KN$ and $\omega\in\Omegaoneg$. Alternatively, as is common
for a general differential graded algebra (not necessarily defined by
derivations), a left connection on $\Omegaoneg$ is a linear map
$\nabla:\Omegaoneg\to\Omegaoneg\otimes_{\KN}\Omegaoneg$ satisfying
\begin{align*}
  \nabla(a\omega) = a\nabla\omega + da\otimes_{\KN}\omega,
\end{align*}
and in case the differential graded algebra is defined from
derivations, one can recover a connection
$\nabla:\g\times\Omegaoneg\to\Omegaoneg$ by evaluating the leftmost
1-form in the tensor product against the derivation.  However, the
converse is not necessarily true, as it depends on the image of the
action of the derivations in $\g$. In fact, for $\KN$ there are no
non-trivial connections
$\nabla:\Omegaoneg\to\Omegaoneg\otimes_{\KN}\Omegaoneg$ due to the
following result.

\begin{lemma}
  $\Omegaoneg\otimes_{\KN}\Omegaoneg=0$.
\end{lemma}

\begin{proof}
  Let us start by noting that
  \begin{align*}
    d\alpha_I\otimesKN d\alpha_J = d\alpha_I\otimesKN ed\alpha_J
    = (d\alpha_I)e\otimesKN d\alpha_J = 0
  \end{align*}
  since $ed\alpha_J=d\alpha_J$ and $(d\alpha_I)e$ = 0 by
  Proposition~\ref{prop:Omega1.bimodule.structure}. Since every element
  $\omega\in\Omegaoneg$ can be written as $\omega=\lambda^Id\alpha_I$
  (where $\lambda^I\in\complex)$, it follows that an arbitrary element
  of $\Omegaoneg\otimesKN\Omegaoneg$ can be written as
  \begin{align*}
    \lambda^{IJ}d\alpha_I\otimesKN d\alpha_J
  \end{align*}
  for $\lambda^{IJ}\in\complex$. Hence, by the previous argument it
  follows that $\Omegaoneg\otimesKN\Omegaoneg=0$.
\end{proof}

\noindent
Note that the above result does not prevent the existence of
connections $\nabla:\g\times\Omegaoneg\to\Omegaoneg$, as we shall see
in the following. Let us start by showing that any bilinear map
$\nabla:\g\times\Omegaoneg\to\Omegaoneg$ defines a connection.

\begin{proposition}\label{prop:C.linear.gives.connection}
  Let $\nabla$ be a $\complex$-bilinear map
  \begin{equation*}
    \nabla:\g\times\Omegaoneg\to\Omegaoneg.
  \end{equation*}
  Then $\nabla$ is a bimodule connection on $\Omegaoneg$. Moreover, if
  $\d\in\g$ then $\nabla_{\d}:\Omegaoneg\to\Omegaoneg$ is a bimodule
  homomorphism.
\end{proposition}

\begin{proof}
  To prove that $\nabla$ is a bimodule connection, one has to check that
  \begin{align}
    &\nabla_{\d}(a\omega) - a\nabla_{\d}\omega - (\d a)\omega = 0\label{eq:left.Leibniz}\\
    &\nabla_{\d}(\omega a) - (\nabla_{\d}\omega)a - \omega(\d a) = 0\label{eq:right.Leibniz}
  \end{align}
  for $\d\in\g$, $\omega\in\Omegaoneg$ and $a\in\KN$.  From
  Proposition~\ref{prop:Omega1.bimodule.structure} it follows that
  $f\omega=\omega f=0$ for all $f\in\KaN$ and $\d\in\g$, implying that
  $(\d a)\omega=\omega(\d a)=0$ since $\d a\in\KaN$ for all
  $a\in\KN$. Thus, for $a=\lambda\mid + \mu e+a^i\alpha_i$, giving
  $a\omega=(\lambda+\mu)\omega$ and $\omega a=\lambda\omega$ for all
  $\omega\in\Omegaoneg$, one finds that
  \begin{align*}
    &\nabla_{\d}(a\omega) - a\nabla_{\d}\omega - (\d a)\omega
    = \nabla_{\d}\paraa{(\lambda+\mu)\omega}
    -(\lambda+\mu)\nabla_{\d}\omega = 0\\
    &\nabla_{\d}(\omega a) - (\nabla_{\d}\omega)a - \omega(\d a)
    = \nabla_{\d}(\lambda\omega)
    -\lambda\nabla_{\d}\omega = 0,
  \end{align*}
  showing that \eqref{eq:left.Leibniz} and \eqref{eq:right.Leibniz}
  are satisfied. Moreover, since $(\d a)\omega=\omega(\d a)=0$ it
  follows that
  \begin{align*}
    \nabla_{\d}(a\omega) = a\nabla_{\d}\omega\qand
    \nabla_{\d}(\omega a) = (\nabla_{\d}\omega)a
  \end{align*}
  showing that $\nabla_{\d}$ is indeed a bimodule homomorphism.
\end{proof}

\noindent
Note in particular that the above result implies that there exists a trivial
connection on $\Omegaoneg$, given by $\nabla_{\d}\omega=0$ for all
$\d\in\g$ and $\omega\in\Omegaoneg$.

Next, let us consider hermitian forms on $\Omegaoneg$. It turns out
that, due to the bimodule relations, the components of any
hermitian form lie in the ideal $\KaN$.

\begin{proposition}\label{prop:hab.in.KaN}
  Let $\{dx_a\}_{a=1}^n$ denote a vector space basis of $\Omegaoneg$
  and let $\omega=\omega^adx_a$, $\eta=\eta^adx_a$, with
  $\omega^a,\eta^a\in\complex$, be arbitrary elements of
  $\Omegaoneg$. The map $h:\Omegaoneg\times\Omegaoneg\to\KN$, defined by
  \begin{equation}\label{eq:def.hform.Omega}
    h(\omega,\eta) = \omega^ah_{ab}\bar{\eta}^b,
  \end{equation}
  is a left hermitian form on $\Omegaoneg$ if and only if
  $h_{ba}=h_{ab}^\ast\in\KaN$ for $a,b=1,\ldots,N$.
\end{proposition}

\begin{proof}
  First, assume that $h_{ba}=h_{ab}\in\KaN$. To show that
  \eqref{eq:def.hform.Omega} defines a left hermitian form on
  $\Omegaoneg$ one needs to check that
  \begin{align*}
    h(\omega,\eta)^\ast = h(\eta,\omega)\qquad
    h(a\omega,\eta) = ah(\omega,\eta)\qquad
  \end{align*}
  for all $\omega,\eta\in\Omegaoneg$ and $a\in\KN$ (the necessary
  bilinearity properties are trivial). It
  follows immediately from \eqref{eq:def.hform.Omega} that
  \begin{align*}
    h(\omega,\eta)^\ast
    &= (\omega^ah_{ab}\bar{\eta}^b)^\ast
      = \eta^bh_{ba}\bar{\omega}^a = h(\eta,\omega)
  \end{align*}
  since $h_{ab}^\ast=h_{ba}$.  Recall that, for $b\in\KaN$ and
  $a=\lambda\mid + \mu e + a^i\alpha_i$, one has
  \begin{align*}
    ab = (\lambda+\mu)b\qand
    ba = \lambda b
  \end{align*}
  giving
  \begin{align*}
    &h(a\omega,\eta)
      = h\paraa{(\lambda+\mu)\omega^adx_a,\eta^bdx_b}
      = (\lambda+\mu)h_{ab}\omega^a\bar{\eta}^b = ah_{ab}\omega^a\bar{\eta}^b
      = ah(\omega,\eta)
  \end{align*}
  using that $a\omega = (\lambda+\mu)\omega$ for all
  $\omega\in\Omegaoneg$. Hence, $h$ is a left hermitian form on
  $\Omegaoneg$.

  Conversely, assume that $h$ is a left hermitian form on
  $\Omegaoneg$. Let us first show that the bimodule relations imply
  that $h_{ab}\in\KaN$. To this end, let us fix $a,b$ and write
  $h_{ab}=\lambda\mid + \mu e + a^i\alpha_i$. Requiring
  $h(edx_a,dx_b)=eh(dx_a,dx_b)$ gives
  \begin{align*}
    h_{ab} = eh_{ab}\implies
    \lambda\mid+\mu e + a^i\alpha_i = (\lambda+\mu)e + a^i\alpha_i\implies\lambda = 0
  \end{align*}
  and, moreover, requiring $h(\alpha_idx_a,dx_b)=\alpha_ih(dx_a,dx_b)$ gives
  \begin{align*}
    0=\alpha_ih_{ab}\implies 0 = \mu e\implies \mu=0
  \end{align*}
  showing that $h_{ab}\in\KaN$ for all $a,b=1,\ldots,N$.
\end{proof}

\noindent
Recall that a hermitian form $h$ is invertible if the
associated map $\hh:\Omegaoneg\to(\Omegaoneg)^\ast$, given by
$\hh(\omega)(\eta) = h(\omega,\eta)$ is a bijection. If $h$ is not
invertible, it is said to be degenerate. As we will see in
Proposition~\ref{prop:no.invertible.forms}, every hermitian form on
$\Omegaoneg$ is degenerate. However, in order to prove this statement,
let us first recall the following result which gives a
characterization of projective modules with invertible hermitian
forms.

\begin{proposition}[\cite{a:levi-civita.class.nms}, Proposition 2.6]\label{prop:huab.equiv.projective}
  Let $M$ be a left $\A$-module with generators $\{e_a\}_{a=1}^n$, let
  $h$ be a left hermitian form on $M$ and set
  $h_{ab}=h(e_a,e_b)$. Then $M$ is a projective module and $h$ is an
  invertible hermitian form if and only if there exist $h^{ab}\in\A$
  (for $a,b=1,\ldots,n$) such that $(h^{ab})^\ast=h^{ba}$ and
  $h_{cp}h^{pq}e_q=e_c$  for $a,b,c=1,\ldots,n$.
\end{proposition}

\noindent
Since $\Omegaoneg$ is a left projective module, we can use the above
result to prove that there are no invertible hermitian forms on
$\Omegaoneg$.

\begin{proposition}\label{prop:no.invertible.forms}
  Every left hermitian form on $\Omegaoneg$ is degenerate.
\end{proposition}

\begin{proof}
  Let $\{dx_a\}_{a=1}^n$ be a vector space basis of $\Omegaoneg$. (In
  particular, $\{dx_a\}_{a=1}^n$ is also a set of generators for
  $\Omegaoneg$.) Since $\Omegaoneg$ is projective,
  Proposition~\ref{prop:huab.equiv.projective} implies that a left
  hermitian form on $\Omegaoneg$ is invertible if and only if there
  exist $h^{ab}\in\KN$ such that
  \begin{align}\label{eq:hab.huab.dxa}
    h_{ab}h^{bc}dx_c = dx_a
  \end{align}
  for $a=1,\ldots,N$ where $h_{ab}=h(dx_a,dx_b)$. Since $\KaN$ is an
  ideal and $h_{ab}\in\KaN$ (by Proposition~\ref{prop:hab.in.KaN}), it
  follows that $h_{ab}h^{bc}\in\KaN$ for all $a,c=1,\ldots,N$. Since
  $adx_c=0$ for all $a\in\KaN$ (by
  Proposition~\ref{prop:Omega1.bimodule.structure}) if follows that
  the left hand side of \eqref{eq:hab.huab.dxa} is zero for any
  $h^{ab}\in\KN$, which contradicts the fact that $\{dx_a\}_{a=1}^n$
  is a basis of $\Omegaoneg$. Thus, we conclude that no such $h^{ab}$
  exist, implying that $h$ is degenerate.
\end{proof}

\noindent
In \cite{a:levi-civita.class.nms} it is shown that one can always
construct a class of connections that are compatible with an
invertible hermitian form on a projective module. However,
Proposition~\ref{prop:no.invertible.forms} clearly implies that one
can not make use of these results. As we shall see, connections
compatible with degenerate hermitian forms can exist as well.

Now, assume that $\{dx_a\}_{a=1}^n$ is a vector space basis of
$\Omegaoneg$. Proposition~\ref{prop:C.linear.gives.connection} implies
that one can define a connection by choosing linear maps
$\gamma_a^b:\g\to\complex$ for $a,b=1,\ldots,n$,
and setting
\begin{align*}
  \nabla_{\d}dx_a = \gamma^b_a(\d)dx_b,
\end{align*}
extending linearly to all of $\Omegaoneg$. Let us now consider the
construction of torsion free connections. By definition, a torsion free
connection satisfies
\begin{align*}
  d\omega(\d_1,\d_2) = \paraa{\nabla_{\d_1}\omega}(\d_2)
  -\paraa{\nabla_{\d_2}\omega}(\d_1)
\end{align*}
for $\omega\in\Omegaoneg$ and $\d_1,\d_2\in\g$. The implications of
this condition on the maps $\gamma_a^b:\g\to\complex$ depend on the
Lie algebra $\g$. Let us point out two cases below.

\begin{lemma}\label{lemma:christoffel.djk}
  Let $\nabla$ be a connection on $\Omegaoneg$ given by
  $\nabla_{\d}d\alpha_I=\gamma_I^J(\d)d\alpha_J$, where
  $\gamma_{I}^J:\g\to\complex$ are $\complex$-linear maps for
  $I,J=0,\ldots,N$.  If $\d_i^j\in\g$ for $i,j=1,\ldots,N$ then
  $\{d\alpha_i\}_{i=1}^N$ are linearly independent over $\complex$
  and, moreover, if $\nabla$ is a
  torsion free connection then $\gamma_I^j(\d_k^l)=0$ for $I=0,\ldots,N$ and $j,k,l=1,\ldots,N$.
\end{lemma}

\begin{proof}
  Assume that $\d_j^k\in\g$ for $j,k=1,\ldots,N$. Let us first show that
  $\{d\alpha_i\}_{i=1}^N$ are linearly independent. To this end,
  assume that
  \begin{align*}
    &\lambda^id\alpha_i=0 \implies
    \lambda^id\alpha_i(\d_j^k) = 0\implies
      \lambda^i\d_j^k\alpha_i=0\implies\\
    &\lambda^i\delta^k_i\alpha_j = 0\implies
      \lambda^k\alpha_j=0
  \end{align*}
  for all $j,k=1,\ldots,N$, implying that $\lambda^k=0$ for
  $k=1,\ldots,N$. Hence, $\{d\alpha_i\}_{i=1}^N$ are linearly independent.

  Since $\nabla$ is a torsion free connection, one finds that
  \begin{align*}
    0 &=(\nabla_{\d_j^k}d\alpha_I)(\d_l^m)
        -(\nabla_{\d_l^m}d\alpha_I)(\d_j^k)
        = \gamma_I^J(\d_j^k)d\alpha_J(\d^l_m)
        -\gamma_I^J(\d_l^m)d\alpha_J(\d_j^k)\\
      &= \gamma_I^m(\d_j^k)\alpha_l
        -\gamma_I^k(\d_l^m)\alpha_j
  \end{align*}
  for all $I=0,1,\ldots,N$ and $j,k,l,m = 1,\ldots,N$. In particular,
  for $j\neq l$ this implies that $\gamma_I^m(\d_j^k) = 0$ for
  $I=0,1,\ldots,N$ and $j,k,m = 1,\ldots,N$.
\end{proof}

\noindent
Thus, for any Lie algebra $\g$ such that $\d_i^j\in\g$ for
$i,j=1,\ldots,N$, torsion free connections on $\Omegaoneg$ satisfy
\begin{align*}
  \nabla_{\d_i^j}d\alpha_I = \gamma_{I}^J(\d_i^j)d\alpha_J
  = \gamma_{I}^0(\d_i^j)d\alpha_0.
\end{align*}

\begin{lemma}\label{lemma:christoffel.dk}
  Let $\nabla$ connection on $\Omegaoneg$ given by
  $\nabla_{\d}d\alpha_I=\gamma_I^J(\d)d\alpha_J$, where
  $\gamma_{I}^J:\g\to\complex$ is a $\complex$-linear map for
  $I,J=0,\ldots,N$. If $\nabla$ is torsion free and $\d_i\in\g$ for
  $i=1,\ldots,N$, then $\gamma_I^0(\d_k)=0$ for $I=0,\ldots,N$ and
  $k=1,\ldots,N$.
\end{lemma}

\begin{proof}
  Note that if $\d_k\in\g$ then $de\neq 0$ since
  $de(\d_k)=i\alpha_k\neq 0$. Assuming that $\nabla$ is torsion free,
  one obtains (for $k\neq j$)
  \begin{align*}
    0 &= \paraa{\nabla_{\d_k}d\alpha_I}(\d_j)-\paraa{\nabla_{\d_j}d\alpha_I}(\d_k)
        = \gamma_I^J(\d_k)d\alpha_J(\d_j)
        -\gamma_I^J(\d_j)d\alpha_J(\d_k)\\
      &= i\gamma_I^0(\d_k)\d_je - i\gamma_I^0(\d_j)\d_ke
        = -\gamma_I^0(\d_k)\alpha_j + \gamma_I^0(\d_j)\alpha_k
  \end{align*}
  which implies that $\gamma_I^0(\d_k)=0$ for $k=1,\ldots,N$ and
  $I=0,\ldots,N$.
\end{proof}

\noindent
In this case, if $\d_i\in\g$ for $i=1,\ldots,N$ then any torsion free
connection on $\Omegaoneg$ satisfy
\begin{align*}
  \nabla_{\d_i}d\alpha_I = \gamma_I^J(\d_i)d\alpha_J
  =\gamma_I^k(\d_i)d\alpha_k.
\end{align*}
Clearly, in the case when $\g=\Der(\KN)$,
Lemma~\ref{lemma:christoffel.djk} and Lemma~\ref{lemma:christoffel.dk}
apply, giving the following result, showing that there is a unique torsion free
connection on $\Omega^1_{\Der}$.

\begin{proposition}\label{prop:torsion.free.on.OmegaDer}
  If $\nabla$ is a torsion free connection on $\Omega_{\Der}^1$
  then $\nabla_{\d}\omega=0$ for all $\d\in\Der(\KN)$ and
  $\omega\in\Omega_{\Der}^1$.
\end{proposition}

\begin{proof}
  Let $\nabla$ be a a torsion free connection on $\Omega_{\Der}^1$ and write
  \begin{align*}
    \nabla_{\d}d\alpha_I = \gamma_I^J(\d)d\alpha_J
  \end{align*}
  where $\gamma_I^J:\g\to\complex$ is a $\complex$-linear map. Since
  $\nabla$ is torsion free and $\d_j^k,\d_l\in\g$ for
  $j,k,l=1,\ldots,N$, one can use Lemma~\ref{lemma:christoffel.djk}
  and Lemma~\ref{lemma:christoffel.dk} to conclude that
  $\gamma_I^0(\d_k)=\gamma_I^j(\d_k^l)=0$ for $j,k,l=1,\ldots,N$ and
  $I=0,\ldots,N$.

  Moreover, since $\nabla$ is torsion free, one finds that
  \begin{align*}
    0 &= \paraa{\nabla_{\d_j}d\alpha_I}(\d_k^l)-\paraa{\nabla_{\d_k^l}d\alpha_I}(\d_j)
        =\gamma_I^J(\d_j)d\alpha_J(\d_k^l)
        -\gamma_I^J(\d_k^l)d\alpha_J(\d_j)\\
      &= \gamma_I^m(\d_j)\delta_m^l\alpha_k-i\gamma_I^0(\d_k^l)\d_je
        =\gamma_I^l(\d_j)\alpha_k+\gamma_I^0(\d_k^l)\alpha_j,
  \end{align*}
  implying that $\gamma_I^k(\d_j)=\gamma_I^0(\d_j^k)=0$ for
  $I=0,\ldots,N$ and $j,k=1,\ldots,N$. Together with
  $\gamma_I^0(\d_k)=\gamma_I^j(\d_k^l)=0$, one concludes that
  $\gamma_I^J(\d_i)=\gamma_I^J(\d_i^j)=0$ for $i,j=1,\ldots,N$ and
  $I,J=0,\ldots,N$. Since $\{\d_i,\d_j^k\}$ is a basis of $\Der(\KN)$
  and $\{d\alpha_I\}_{I=0}^N$ is a basis of $\Omega_{\Der}^1$, one concludes
  that $\nabla_{\d}\omega=0$ for all $\d\in\g$ and
  $\omega\in\Omega_{\Der}^1$.
\end{proof}

\section{Levi-Civita connections on $\Omegaoneg$}
\label{sec:LC.Omegaoneg}

Given a left hermitian form $h$ on $\Omegaoneg$, a (left) Levi-Civita connection
is a connection on $\Omegaoneg$ such that
\begin{align*}
  &\d h(\omega,\eta) = h\paraa{\nabla_{\d}\omega,\eta} + h(\omega,\nabla_{\d^\ast}\eta)\\
  &(\nabla_{\d}\omega)(\d')-(\nabla_{\d'}\omega)(\d)-d\omega(\d,\d') = 0
\end{align*}
for $\d,\d'\in\g$ and $\omega,\eta\in\Omegaoneg$. If $\g=\Der(\KN)$
then Proposition~\ref{prop:torsion.free.on.OmegaDer} implies that a
torsion free connection on $\Omega^1_{\Der}$ is compatible with a
hermitian form $h$ if and only if $\d h(\omega,\eta)=0$ for all
$\d\in\Der(\KN)$ and $\omega,\eta\in\Omega^1_{\Der}$, giving
$h(\alpha_I,\alpha_J)=\lambda_{IJ}\mid$ for some
$\lambda_{IJ}\in\complex$. However, Proposition~\ref{prop:hab.in.KaN}
then implies that $\lambda_{IJ}=0$ since $h(d\alpha_I,d\alpha_J)$ is
necessarily in $\KaN$. Hence, it is only for the trivial hermitian
form, i.e $h(\omega,\eta)=0$ for all $\omega,\eta\in\Omega^1_{\Der}$,
that a Levi-Civita connection exists on $\Omega^1_{\Der}$. Thus, to
get more interesting examples of Levi-Civita connections one needs to
choose a different Lie algebra of derivations. In the following, we
will present a few examples where Levi-Civita connections exist for
non-trivial choices of hermitian form.

\subsection{Levi-Civita connections for inner derivations}
Let us now consider the case when $\g=\ginn$ is the Lie algebra of inner
derivations spanned by $\dh=\d_1^1+\cdots+\d_N^N$ and $\d_i$ for $i=1,\ldots,N$, satisfying
\begin{align*}
  [\d_i,\d_j]=0\qand
  [\dh,\d_i] = \d_i
\end{align*}
for $i,j=1,\ldots,N$, and
\begin{align*}
  \d_k(a) = -i[\alpha_k,a]\qand
  \dh(a) = [e,a]
\end{align*}
for all $a\in\KN$, as well as $\dh\alpha_i=\alpha_i$ for
$i=1,\ldots,N$. The calculus is connected since for
$a=\lambda\mid + \mu e + a^i\alpha$, $da=0$ implies that
\begin{align*}
  &da(\d_i) = 0\implies \d_ia = 0 \implies \mu\alpha_i = 0\implies\mu = 0\\
  &da(\dh) = 0 \implies \dh a = 0\implies a^i\alpha_i = 0 \implies a^i = 0
\end{align*}
giving $a=\lambda\mid$. Thus, by
Proposition~\ref{prop:connected.vspace.basis}, we know that
$\{d\alpha_I\}_{I=0}^N$ is a vector space basis of $\Omegaoneg$.
An arbitrary connection on $\Omegaoneg$ is given by
\begin{align}\label{eq:conn.inner.gamma}
  \nabla_{\d}d\alpha_I = \gamma_I^J(\d)d\alpha_J
\end{align}
where $\gamma_I^J:\g\to\complex$ is a $\complex$-linear map. For inner
derivations, there is a canonical connection given by
\begin{align*}
  \nabla_{[a,\cdot]}\omega = [a,\omega]
\end{align*}
corresponding to $\gamma_I^J(\d_i)=0$ and $\gamma_I^J(\dh)=\delta_I^J$
in \eqref{eq:conn.inner.gamma}. However, this connection is not
torsion free since
\begin{align*}
  \paraa{\nabla_{[a,\cdot]}d\alpha_I}([b,\cdot]) -\paraa{\nabla_{[b,\cdot]}d\alpha_I}([a,\cdot])
  &= [a,\alpha_I]([b,\cdot])-[b,\alpha_I]([a,\cdot])\\
  &= [b,[a,\alpha_I]]-[a,[b,\alpha_I]]
    =[\alpha_I,[a,b]]
\end{align*}
giving, for $I=0$, $a=e$ and $b=\alpha_k$,
\begin{align*}
  \paraa{\nabla_{[e,\cdot]}de}([\alpha_k,\cdot]) -\paraa{\nabla_{[\alpha_k,\cdot]}de}([e,\cdot])
  =[e,[e,\alpha_k]] = [e,\alpha_k] = \alpha_k \neq 0.
\end{align*}

\begin{proposition}\label{prop:inner.der.torsion.free}
  $\nabla$ is a torsion free connection on $\Omegaoneg$ if and only if
  there exist $\gamma_I^J\in\complex$, for $I,J=0,\ldots,N$ such that
  \begin{align}
    \nabla_{\dh}d\alpha_I &= \gamma_I^Jd\alpha_J\label{eq:inner.nabla.torsion.two}\\
    \nabla_{\d_k}d\alpha_I &= -\gamma_I^0d\alpha_k\label{eq:inner.nabla.torsion.one}
  \end{align}
  for $k=1,\ldots,N$ and $I=0,\ldots,N$.
\end{proposition}

\begin{proof}
  Let $\nabla$ be a a torsion free connection on $\Omegaoneg$ and write
  \begin{align*}
    \nabla_{\d}d\alpha_I = \gamma_I^J(\d)d\alpha_J
  \end{align*}
  where $\gamma_I^J:\g\to\complex$ is a $\complex$-linear map. Since
  $\nabla$ is torsion free, one obtains (for $k\neq j$)
  \begin{align*}
    0 &= \paraa{\nabla_{\d_k}d\alpha_I}(\d_j)-\paraa{\nabla_{\d_j}d\alpha_I}(\d_k)
        = \gamma_I^J(\d_k)d\alpha_J(\d_j)
        -\gamma_I^J(\d_j)d\alpha_J(\d_k)\\
      &= i\gamma_I^0(\d_k)\d_je - i\gamma_I^0(\d_j)\d_ke
        = -\gamma_I^0(\d_k)\alpha_j + \gamma_I^0(\d_j)\alpha_k
  \end{align*}
  which implies that $\gamma_I^0(\d_k)=0$ for $k=1,\ldots,N$ and
  $I=0,\ldots,N$. Furthermore, one finds that
  \begin{align*}
    0 &= \paraa{\nabla_{\d_k}d\alpha_I}(\dh)-\paraa{\nabla_{\dh}d\alpha_I}(\d_k)
        =\gamma_I^J(\d_k)d\alpha_J(\dh)
        -\gamma_I^J(\dh)d\alpha_J(\d_k)\\
      &= \gamma_I^j(\d_k)\alpha_j
        -i\gamma_I^0(\dh)de(\d_k)
       = \gamma_I^j(\d_k)\alpha_j
        +\gamma_I^0(\dh)\alpha_k
  \end{align*}
  giving $\gamma_I^j(\d_k)=0$ for $j\neq k$ and
  $\gamma_I^k(\d_k)=-\gamma_I^0(\dh)$ for $I=0,\ldots,N$ and
  $j,k=1,\ldots,N$. Thus, one gets
  \begin{align*}
    &\nabla_{\dh}d\alpha_I = \gamma_I^J(\dh)d\alpha_J\\
    &\nabla_{\d_k}d\alpha_I = \gamma_I^J(\d_k)d\alpha_J
      =\gamma_I^k(\d_k)d\alpha_k
      =-\gamma_I^0(\dh)d\alpha_k
  \end{align*}
  giving \eqref{eq:inner.nabla.torsion.one} and
  \eqref{eq:inner.nabla.torsion.two} with
  $\gamma_I^J=\gamma_I^J(\dh)$. Conversely, it is straightforward to
  check that the connection defined by
  \eqref{eq:inner.nabla.torsion.one} and
  \eqref{eq:inner.nabla.torsion.two} is torsion free.
\end{proof}

\noindent
Since $\d_k^\ast=\d_k$, $\dh^\ast=\dh$ and
$(d\alpha_I)^\ast=d\alpha_I$, it is easy to check that the torsion
free connections in Proposition~\ref{prop:inner.der.torsion.free} are
$\ast$-connections if $\gamma_I^J\in\reals$.

Let us now construct a torsion free $\ast$-connection compatible with
the left hermitian form of the type
\begin{align*}
  \paraa{h_{IJ}} = \paraa{h(d\alpha_I.d\alpha_J)} =
  \begin{pmatrix}
    \rho_0 & -i\rho_1 & \cdots & -i\rho_N\\
    i\rho_1 & \\
    \vdots & & 0 & \\
    i\rho_N
  \end{pmatrix}h_0
\end{align*}
for $\rho_I\in\reals$ for $I=0,\ldots,N$ and $h_0\in\KaN$ with
$h_0^\ast=h_0$. Note that $h_{IJ}$ defines a left hermitian form via
\begin{align*}
  h(\omega^Id\alpha_I,\eta^J d\alpha_J) =
  \omega^Ih_{IJ}\bar{\eta}^J
\end{align*}
by Proposition~\ref{prop:hab.in.KaN}. Let us now choose
$\gamma^I_J=\tfrac{1}{2}\delta_J^I$ in
Proposition~\ref{prop:inner.der.torsion.free}, giving a
torsion free $\ast$-connection on $\Omegaoneg$ as
\begin{align*}
  \nabla_{\dh}d\alpha_I = \tfrac{1}{2}d\alpha_I\qquad
  \nabla_{\d_k}d\alpha_0 = -\tfrac{1}{2}d\alpha_k\qquad
  \nabla_{\d_k}d\alpha_l = 0
\end{align*}
for $I=0,\ldots,N$ and $k,l=1,\ldots,N$. Let us now check metric compatibility of $\nabla$:
\begin{align*}
  \dh h(d\alpha_I,d\alpha_J) - h(\nabla_{\dh}d\alpha_I,d\alpha_J)-h(d\alpha_I,\nabla_{\dh}d\alpha_J)
  = \dh h_{IJ} -\tfrac{1}{2}h_{IJ}-\tfrac{1}{2}h_{IJ} = 0
\end{align*}
since $\dh h_{IJ} = h_{IJ}$, and furthermore
\begin{align*}
  \d_k h(d\alpha_I,d\alpha_J) - h(\nabla_{\d_k}d\alpha_I,d\alpha_J)-h(d\alpha_I,\nabla_{\d_k}d\alpha_J)
  = \tfrac{1}{2}\delta_I^0h_{kJ} + \tfrac{1}{2}\delta_J^0h_{Ik}
\end{align*}
If $I,J\neq 0$ then the above expression clearly vanishes, and if $I=J=0$ then one has
\begin{align*}
  \tfrac{1}{2}\delta_I^0h_{kJ} + \tfrac{1}{2}\delta_J^0h_{Ik}
  =\tfrac{1}{2}h_{k0} + \tfrac{1}{2}h_{0k}
  =\tfrac{1}{2}i\rho_k -\tfrac{1}{2}i\rho_k = 0. 
\end{align*}
If $I=0$ and $J=j\neq 0$ (and similarly for $J=0$ and $I\neq 0$) then 
\begin{align*}
  \tfrac{1}{2}\delta_I^0h_{kJ} + \tfrac{1}{2}\delta_J^0h_{Ik}
  = \tfrac{1}{2}h_{kj} = 0
\end{align*}
since $h_{kj}=0$. Hence, $\nabla$ is a $\ast$-bimodule Levi-Civita
connection compatible with $h$.

\subsection{Levi-Civita connections for a set of outer derivations}

In this section, let
\begin{align*}
  \g = \complex\angles{\dt_i=\d_i+\d_i^i:\,\, i=1,\ldots,N}
\end{align*}
giving $\dt_i^\ast=\dt_i$, $[\dt_i,\dt_j]=0$ and
\begin{align*}
  \dt_ke = i\alpha_k\qand\dt_i\alpha_k = \delta_{ik}\alpha_i.
\end{align*}
Note that $\dt_i$ is an outer derivation for $i=1,\ldots,N$, but
\begin{align*}
  \dt_1+\cdots+\dt_N = \d_1+\cdots+\d_N+\d_1^1+\cdots+\d_N^N
\end{align*}
is a sum of inner derivations, and therefore an inner derivation.

\begin{proposition}
  A vector space basis for $\Omegaoneg$ is given by
  $\{d\alpha_i\}_{i=1}^N$ and it holds that
  \begin{align*}
    d\alpha_0 + d\alpha_1+\cdots+d\alpha_N=0
  \end{align*}
  It follows that $\Omegaoneg$ is not connected.
\end{proposition}

\begin{proof}
  One computes
  \begin{align*}
    -d\alpha_0(\dt_k) = -i(\d_k+\d^k_k)(e) = -i\d_ke = -i(i\alpha_k)=\alpha_k
  \end{align*}
  and
  \begin{align*}
    (d\alpha_1+\cdots+d\alpha_N)(\dt_k)
    = \d_k^k(\alpha_1)+\cdots+\d_k^k(\alpha_N)
    = \d_k^k\alpha_k = \alpha_k, 
  \end{align*}
  showing that $d\alpha_0+d\alpha_1+\cdots+d\alpha_k = 0$. It follows
  immediately that
  \begin{align*}
    d(\alpha_1+\cdots+\alpha_N+ie) = 0 
  \end{align*}
  which shows that $\Omegaoneg$ is not connected. Moreover, it follows
  that $\{d\alpha_k\}_{k=1}^N$ generates $\Omegaoneg$. Let us now show
  that $\{d\alpha_k\}_{k=1}^N$ are linearly independent over
  $\complex$. Thus, assume that $\lambda^kd\alpha_k=0$ for $\lambda^k\in\complex$, which implies
  that
  \begin{align*}
    0=\lambda^kd\alpha_k(\dt_l) = \lambda^k\d_l^l\alpha_k
    =\lambda^l\alpha_l\quad\text{(no sum over $l$)} 
  \end{align*}
  giving $\lambda^l=0$ for $l=1,\ldots,N$.
\end{proof}

\noindent
Thus, it follows
(cf. Proposition~\ref{prop:C.linear.gives.connection}) that an
arbitrary connection on $\Omegaoneg$ is determined by
\begin{align*}
  \nabla_{\dt_i}d\alpha_j = \gamma_{ij}^kd\alpha_k
\end{align*}
with $\gamma_{ij}^k\in\complex$.

\begin{proposition}\label{prop:connection.dti.torsion.free}
  A connection $\nabla:\g\times\Omegaoneg\to\Omegaoneg$ is torsion
  free if and only if there exists $\gamma_{ij}\in\complex$ such that
  \begin{align}\label{eq:dt.torsion.free.conn}
    \nabla_{\dt_i}d\alpha_j = \gamma_{ij}d\alpha_i
  \end{align}
  for $i,j=1,\ldots,N$.
\end{proposition}

\begin{proof}
  One computes (writing out sums explicitly)
  \begin{align*}
    \paraa{\nabla_{\dt_k}d\alpha_i}(\dt_l)-\paraa{\nabla_{\dt_l}d\alpha_i}(\dt_k)
    &=\sum_{m=1}^N\bracketb{\gamma_{ki}^m\dt_l(\alpha_m)-\gamma_{li}^m\dt_k(\alpha_m)}
    =\gamma_{ki}^l\alpha_l -  \gamma_{li}^k\alpha_k,
  \end{align*}
  for $k,l=1,\ldots,N$, and concludes that the connection is torsion free if
  $\gamma_{ki}^l=0$ for $i=1,\ldots,N$ and $k\neq l$. Hence, a torsion
  free connection is given by
  \begin{align*}
    \nabla_{\dt_i}d\alpha_j = \gamma_{ij}^id\alpha_i,
  \end{align*}
  showing that every torsion free connection is given as in
  \eqref{eq:dt.torsion.free.conn} (setting
  $\gamma_{ij}=\gamma_{ij}^i$).
\end{proof}

\noindent
Since $\dt_i^\ast=\dt_i$ and $(d\alpha_i)^\ast=d\alpha_i$ it is easy
to check that the connection in
Proposition~\ref{prop:connection.dti.torsion.free} is a
$\ast$-connection if $\gamma_{ij}\in\reals$.  Now, let us construct a
torsion free $\ast$-connection on $\Omegaoneg$ that is compatible with
the diagonal hermitian form given by
\begin{align*}
  h_{ij}=h(d\alpha_i,d\alpha_j)=\delta_{ij}\lambda_i\alpha_i 
\end{align*}
for $\lambda_i\in\reals$. Setting $\gamma_{ij}=\tfrac{1}{2}\delta_{ij}$ in
Proposition~\ref{prop:connection.dti.torsion.free} it follows that
\begin{align}\label{eq:dti.lc.connection}
  \nabla_{\dt_i}d\alpha_j = \tfrac{1}{2}\delta_{ij}d\alpha_i
\end{align}
is a torsion free $\ast$-connection. Next, one checks that $\nabla$ is
compatible with $h$:
\begin{align*}
  \dt_ih_{jk}
  &- h(\nabla_{\dt_i}d\alpha_j,d\alpha_k)-h(d\alpha_j,\nabla_{\dt_i}d\alpha_k)\\
  & = \delta_{jk}\lambda_j\dt_id\alpha_j
    - \tfrac{1}{2}\delta_{ij}h_{ik} - \tfrac{1}{2}\delta_{ik}h_{ji}\\
  &= \lambda_j\delta_{jk}\delta_{ij}\alpha_i
    - \tfrac{1}{2}\delta_{ij}\delta_{ik}\lambda_i\alpha_i
    - \tfrac{1}{2}\delta_{ik}\delta_{ji}\lambda_j\alpha_j = 0.
\end{align*}
Hence, \eqref{eq:dti.lc.connection} defines a $\ast$-bimodule Levi-Civita connection
on $\Omegaoneg$ with respect to the hermitian form
$h(d\alpha_i,d\alpha_j)=\delta_{ij}\lambda_i\alpha_i$.

\section*{Acknowledgements}
J.A. would like to thank E. Darp\"o for discussions. Furthermore,
J.A. is supported by the Swedish Research Council grant 2017-03710.

\bibliographystyle{alpha}
\bibliography{references}  

\end{document}